\documentclass[11pt,reqno]{amsart}

\usepackage{xifthen}

\usepackage{amsmath}
\usepackage{amssymb, latexsym, amsfonts, amscd, amsthm, mathrsfs}
\usepackage[usenames,dvipsnames]{color}
\usepackage[all]{xy}
\usepackage{graphicx}
\usepackage{epsfig}
\usepackage[colorlinks=true,linkcolor=blue]{hyperref}
\usepackage{amssymb}
\usepackage{enumerate}
\usepackage{bm}

\numberwithin{equation}{section}

\definecolor{purple}{rgb}{0.9,0,0.8}

\definecolor{gray}{rgb}{0.7,0.7,0.7}

\newtheorem{thm}{Theorem}[section]
\newtheorem{cor}[thm]{Corollary}
\newtheorem{lem}[thm]{Lemma}
\newtheorem{prop}[thm]{Proposition}

\newtheorem{dfn}[thm]{Definition}

\theoremstyle{definition}

\newtheorem{rmk}[thm]{Remark}
\newtheorem*{rmk*}{Remark}

\newcommand{\beq}{\begin{equation}}
	\newcommand{\eeq}{\end{equation}}

\newcommand{\N}{\mathbb{N}}

\renewcommand{\setminus}{\backslash}

\usepackage{todonotes}
\usepackage{bbm}
\usepackage[utf8]{inputenc}

\newcommand{\eqn}[1]{\begin{equation}#1\end{equation}}
\newcommand{\eqan}[1]{\begin{align}#1\end{align}}
\newcommand{\sss}{\scriptscriptstyle}
\newcommand{\betanV}{\beta_{n,{\sss V}}}
\newcommand{\betanS}{\beta_{n,{\sss S}}}
\newcommand{\betanT}{\beta_{n,{\sss T}}}
\newcommand{\constant}{a}
\newcommand{\ellalt}{\ell}
\newcommand{\TG}{T_{\sss G}}

\newcommand{\psin}{\psi_n}
\newcommand{\nn}{\nonumber}
\newcommand{\dist}{{\rm dist}}
\newcommand{\Gstar}{\mathscr{G}^\star}

\begin{document}

\title{Sparse random graphs with many triangles}

\author[S.\ Chakraborty]{Suman Chakraborty}
\address{Department of Mathematics and Computer Science,  Eindhoven University of Technology, Eindhoven, Netherlands}
\email{contact@sumanc.com}

\author[R.\ van der Hofstad]{Remco van der Hofstad}
\address{Department of Mathematics and Computer Science,  Eindhoven University of Technology, Eindhoven, Netherlands}
\email{rhofstad@win.tue.nl}

\author[F. \ den Hollander]{Frank den Hollander}
\address{Mathematical Institute,  Leiden University, Leiden, Netherlands}
\email{denholla@math.leidenuniv.nl}

\date{\today}


\begin{abstract}
In this paper we consider the Erd\H{o}s-R\'enyi random graph in the sparse regime in the limit as the number of vertices $n$ tends to infinity. We are interested in what this graph looks like when it contains many triangles, in two settings. First, we derive asymptotically sharp bounds on the probability that the graph contains a large number of triangles. We show that, conditionally on this event, with high probability the graph contains an almost complete subgraph, i.e., the triangles form a near-clique, and has the {\em same} local limit as the original Erd\H{o}s-R\'enyi random graph. Second, we derive asymptotically sharp bounds on the probability that the graph contains a large number of vertices that are part of a triangle. If order $n$ vertices are in triangles, then the local limit (provided it exists) is {\em different} from that of the Erd\H{o}s-R\'enyi random graph. Our results shed light on the challenges that arise in the description of real-world networks, which often are {\em sparse, yet highly clustered}, and on exponential random graphs, which often are used to model such networks.  
\end{abstract}

\maketitle


\section{Introduction and main results}
\label{S1}


\subsection{Background and motivation}
\label{back}

Let $G_{n,p_n}$ be the Erd\H{o}s-R\'{e}nyi random graph on $[n]:=\{1,\ldots,n\}$ vertices, where each edge is present with probability $p_n$ independently of each other. In this article we consider the sparse regime: $p_n=\lambda/n$ for some fixed $\lambda>0$. (We call a graph {\em sparse} when the expected number of edges is of the same order as the number of vertices in the graph.)

\medskip
\paragraph{\bf Triangles.}
It is well-known that a typical outcome of $G_{n,p_n}$ with $p_n=\lambda/n$ has roughly $\lambda n/2$ many edges and is locally tree-like, i.e., for each fixed $r\in \N$, the proportion of vertices whose $r$-neighbourhood contains a cycle vanishes with high probability (see, for example, Bollob\'as \cite{bollobas2001random}, Frieze and Karo\'nski \cite{frieze2016introduction}, Janson, Ruci{\'n}ski and  Luczak \cite{janson2011random}, van der Hofstad \cite{van2017random}). Here `with high probability' means with probability tending to $1$ as $n\to\infty$. On the other hand, many networks observed in practice are far from locally tree-like (see Newman \cite{newman2009random} and the references therein). In particular, they often contain a large number of triangles. This motivates our first question: \emph{How unlikely is it to observe a large number of triangles in a typical outcome of $G_{n,p_n}$, and what does the graph looks like if we condition $G_{n,p_n}$ to have a large number of triangles?}

In this article we derive asymptotically sharp upper and lower bounds on the logarithmic probability that $G_{n,p_n}$ contains a large number of triangles. Furthermore, we show that if we condition $G_{n,p_n}$ to have a large number of triangles, then with high probability it contains a small `almost-complete' subgraph, while a typical outcome is still locally tree-like. This is rather disappointing, as in practice a substantial number of vertices of the graph are part of some triangle (Newman \cite{newman2009random}, Rapoport \cite{rapoport1948cycle}, Serrano and Bogu{\~n}a \cite{serrano2006clustering}, Watts and Strogatz \cite{watts1998collective}). This motivates our second question: \emph{What is the probability that a typical outcome of $G_{n,p_n}$ contains a large number of vertices that are part of some triangle?} We derive sharp upper and lower bounds on the logarithmic upper tail probability of this event as well. Although we have not been able to prove structural results for this setting, our proof of upper and lower bounds is based on the Ansatz that if $S$ is the subset of vertices of $G_{n,p_n}$ that are part of some triangle, then $S$ contains a subgraph that has approximately $|S|/3$ vertex-disjoint triangles.    

\medskip
\paragraph{\bf Subgraphs.}
For a small fixed graph $H$, let $N(H,G)$ be the number of copies of $H$ in $G$. For various values of $p_n$, the distribution of $N(H,G_{n,p_n})$ has been studied extensively. The first results appeared in the seminal paper by Erd\H{o}s and R\'{e}nyi \cite{erdos1960evolution}. For $H$ a small complete graph, the probability of $\{N(H,G_{n,p_n})>0\}$ was studied by Sch{\"u}rger \cite{schurger1979limit}. This work was later extended to other subgraphs by Bollob\'{a}s in \cite{bollobas1981threshold}. Ruci\'{n}ski \cite{rucinski1988small} obtained a necessary and sufficient condition for {\em asymptotic normality} of $N(H,G_{n,p_n})$. 

Sharp estimates of probabilities of the form $\mathbb{P}\left(N(H,G_{n,p_n}) \leq (1-\varepsilon)\mathbb{E}[N(H,G_{n,p_n})]\right)$ were obtained in Janson, Ruci{\'n}ski, and Luczak \cite{janson2011random}. For a long time, no sharp estimate was available on `upper tail' probabilities of the form $\mathbb{P}\left(N(H,G_{n,p_n}) \geq (1+\varepsilon)\mathbb{E}[N(H,G_{n,p_n})]\right)$. When $p \in (0,1)$ is independent of $n$, asymptotically sharp estimates of $\log{\mathbb{P}\left(N(H,G_{n,p}) \geq (1+\varepsilon)\mathbb{E}[N(H,G_{n,p})]\right)}$ were obtained by Chatterjee and Varadhan \cite{chatterjee2011large} using the theory of dense graph limits \cite{lovasz2012large}, and Szemer\'{e}di’s regularity lemma \cite{szemeredi1975regular}. Chatterjee and Varadhan \cite{chatterjee2011large}, and Lubetzky and Zhao \cite{lubetzky2015replica} studied the structure of $G_{n,p}$ conditioned on large subgraph counts as well. Due to the poor quantitative estimates in Szemer\'{e}di’s regularity lemma, the arguments in \cite{chatterjee2011large} did not extend easily when $p=p_n$ tends to zero as $n\to \infty$ faster than a negative power of $\log{n}$ (even extending the arguments in \cite{chatterjee2011large} to the regime when $p_n$ is a negative power of $\log{n}$ required a weaker version of the Szemer\'{e}di’s regularity lemma; see \cite{lubetzky2017variational}). 
  
In the regime where $(\log n)/n \ll p_n \ll1$, initial attempts using standard concentration inequalities (such as Hoeffding \cite{hoeffding1994probability}, Talagrand \cite{talagrand1995concentration}, Kim and Vu \cite{kim2000concentration}) yielded estimates that were far from optimal. The estimates obtained with these approaches appeared in the article appropriately titled `The infamous upper tail' by Janson and Ruci{\'n}ski \cite{janson2002infamous}. In Janson, Oleszkiewicz, and Ruci{\'n}ski \cite{janson2004upper}, an upper and a lower bound on $\log{\mathbb{P}\left(N(H,G_{n,p_n}) \geq (1+\varepsilon)\mathbb{E}[N(H,G_{n,p_n})]\right)}$ were obtained that differed by a multiplicative factor $\log(1/p_n)$.

When $H=K_r$ is an $r$-clique (complete subgraph on $r$ vertices), the mis-matching factor $\log(1/p_n)$ was removed by Chatterjee \cite{chatterjee2012missing} and by DeMarco and Kahn \cite{demarco2012tight}, who proved the correct order of magnitude of the logarithmic upper tail probability. Finally, after a series of improvements (Augeri \cite{augeri2018nonlinear}, Chatterjee and Dembo \cite{chatterjee2016nonlinear}, Cook and Dembo \cite{cook2020large}, Eldan \cite{eldan2018gaussian}), Harel, Mousset, and Samotij \cite{harel2019upper} approximated $\log{\mathbb{P}\left(N(K_r,G_{n,p_n})\right)}$ in terms of a variational problem for $(\log n)/n \ll p_n \ll1$. The variational problem was independently solved by Lubetzky and Zhao \cite{lubetzky2017variational} yielding the estimate
	\begin{equation}
	\label{eqn:239pm09sep21}
	\log{\mathbb{P}\left(N(K_r,G_{n,p_n}) \geq (1+\varepsilon)\mathbb{E}[N(K_r,G_{n,p_n})]\right)} \approx -C(\varepsilon)n^2p_n^2 \log(1/p_n),
	\end{equation}
for $(\log n)/n \ll p_n \ll1$, where $C(\varepsilon)>0$ is a constant. 

In their attempt to prove estimates of the form \eqref{eqn:239pm09sep21}, Chatterjee and Dembo \cite{chatterjee2016nonlinear} initiated the study of {\em non-linear large deviations}, i.e., large deviations of non-linear functions of i.i.d.\ Bernoulli random variables. Since then the study of non-linear large deviations has received vast attention (see Augeri \cite{augeri2018nonlinear}, Basak and Basu \cite{basak2019upper}, Chatterjee and Dembo \cite{chatterjee2016nonlinear}, Cook and Dembo \cite{cook2020large}, Eldan \cite{eldan2018gaussian}, Harel, Mousset, and Samotij \cite{harel2019upper}). These works also consider graphs other than $K_r$, and the related variational problem was solved in Bhattacharya, Ganguly, Lubetzky, and Zhao \cite{bhattacharya2017upper}.  

To our knowledge, none of the above works considers the sparse regime $p=p_n=\lambda/ n$, which is the most challenging and is treated in the present paper.

\medskip
\paragraph{\bf Back to triangles.}
It is well-known that, in the sparse regime, in $G_{n,p_n}$ the number of triangles $T=N(K_3,G_{n,p_n})$ converges in law to a Poisson random variable with mean $\lambda^3/6$. In our first main result, we derive sharp upper and lower bounds on the probability of rare events of the form $\{T\geq k_n\}$, where $(k_n)_{n\in\N}$ is a sequence of positive real numbers that tend to infinity. Roughly, we show that, as long as $k_n$ is growing fast enough,
	\begin{equation}
	\label{eqn:407pm09sep21}
	\log\mathbb{P}(T \geq k_n)\approx - \tfrac12[6k_n]^{2/3} \log(1/p_n).
	\end{equation} 
The {\em lower bound} on $\mathbb{P}(T \geq k_n)$ expressed by \eqref{eqn:407pm09sep21} comes from the observation that if $[(6k_n)^{1/3}]$ forms a clique, then we must have $T\geq k_n$. The probability that $[(6k_n)^{1/3}]$ forms a clique is roughly $\exp{\left(- \tfrac12[6k_n]^{2/3} \log(1/p_n)\right)}$. The matching {\em upper bound} expressed by \eqref{eqn:407pm09sep21} is more difficult, and to derive it we adapt the theory of non-linear large deviations developed in Harel, Mousset, and Samotij \cite{harel2019upper}. As mentioned before, the study of small subgraphs is a classical topic in random graph theory. In particular, the estimate \eqref{eqn:407pm09sep21} is a new contribution towards the understanding of the rare event that `an unusually large number of triangles appears in a sparse random graph'.

In this regard, Collet and Eckmann \cite{collet2002number} considered a closely related problem. We need some notation to discuss their result. For $\alpha, \beta>0$, let  $\mathcal{G}(n,\alpha, \beta)$ be the set of graphs on $n$ vertices, $\alpha n$ edges, and $\beta n$ triangles, and let $\mathcal{G}(n,\alpha)$ be the set of graphs on $n$ vertices, and $\alpha n$ edges. Collet and Eckmann \cite{collet2002number} showed that $\log{|\mathcal{G}(n,\alpha, \beta)|}\approx \log{|\mathcal{G}(n,\alpha)|}$. Note that when $\alpha$ is independent of $n$, $\log{|\mathcal{G}(n,\alpha)|}$ is of order $n\log{n}$. This suggests that it might be possible to show that $-\log\mathbb{P}(T \geq \beta n) = o(n \log{n})$ but we were unable to improve this further using the techniques in \cite{collet2002number}, whereas, using \eqref{eqn:407pm09sep21} we can obtain much sharper estimate, that is, 
$$
\log\mathbb{P}(T \geq \beta n) \approx - \tfrac12[6\beta n]^{2/3} \log(n).
$$ 
Motivated by their results the authors in \cite{collet2002number} suggested that ``the triangles seem to cluster even at low density". As far as we know the present paper is the first one to provide rigorous, and quantitative, evidence to this claim in the sparse regime. We describe these in the following paragraph.       

Equation \eqref{eqn:407pm09sep21} provides a rough answer to the question how unlikely it is that a sparse graph with a large number of triangles is actually an outcome of a sparse Erd\H{o}s-R\'{e}nyi graph. We also provide a description of the random graph {\em conditionally} on the event $\{T\geq k_n\}$. Roughly, this asserts that with high probability $G_{n,p_n}$ contains a subgraph $G$ with approximately $(6k_n)^{2/3}/2$ edges with minimum degree approximately $(6k_n)^{1/3}$, which suggests that $G$ is an `almost complete' subgraph. Although this provides structural information of $G_{n,p_n}$ conditionally on $\{T\geq k_n\}$, we also show that this does not affect the {\em local structure} of the graph. Indeed, conditionally on $\{T\geq k_n\}$, $G_{n,p_n}$ still converges locally in probability to the same Galton-Watson tree as the original Erd\H{o}s-R\'{e}nyi random graph. Unfortunately, this is generally not the case in applications (for example, Newman \cite{newman2009random} proposed a random graph model where each vertex $i \in [n]$ participates in $t_i$-triangles and $s_i$-edges that are not part of any triangle). 

The above failure motivates us to study the number of vertices that are part of some triangle in $G_{n,p_n}$, denoted  by $V_{\sss T}(G_{n,p_n})$. We show that, for appropriately growing $k_n$,
	\begin{equation}
	\label{eqn:415pm10sep21}
	{\log \mathbb{P}(V_{\sss T}(G_{n,p_n})\geq k_n)} \approx-\frac{1}{3}k_n \log k_n.
	\end{equation}   
As far as we know, this estimate is the first of its kind. To see where it comes from, note that if, in a graph $G$, there is a set $S \subseteq [n]$ with $|S|=k_n$ such that $S$ is an union of vertex-disjoint triangles, no other edge is present in $[S]$, and $V\setminus S$ is triangle free, then clearly $V_{\sss T}(G) \geq k_n$, and it is relatively straightforward to obtain a lower bound on $\mathbb{P}(V_{\sss T}(G_{n,p_n})\geq k_n)$. However, $V_{\sss T}(G)$ does not seem to be a bounded-degree polynomial function of the entries of the adjacency matrix of $G$. Therefore, unlike $T$, the theory of non-linear large deviations in Harel, Mousset, and Samotij \cite{harel2019upper} is not readily applicable in this case.  To derive an upper bound we instead develop some new combinatorial tools, in particular, a {\em graph decomposition lemma} that we believe to be of independent interest. Shortly after we posted this paper on arXiv, Ganguly, Hiesmayr, and Nam \cite{ganguly2024upper} published a paper addressing the same problem. In particular, while our paper estimates the probability $\mathbb{P}(V_{\sss T}(G_{n,p_n}) \geq k_n)$ in the regime $k_n \geq (\log n)^3$, their paper estimates this probability in the regime $k_n \leq n^{1/10-\delta}$ for some $\delta > 0$. As noted in \cite{ganguly2024upper}, our work and theirs together settle the long-standing ``upper tail problem'' for the number of triangles in the case where the average degree is constant.

\medskip
\paragraph{\bf Real-world networks.} 
Large-number-of-triangles events have practical importance. Many networks observed in practice are sparse, while at the same time they exhibit transitivity (also called clustering), in the sense that two neighbors of the same vertex are more likely to also be neighbors of one another. (See Newman \cite{newman2009random}, Rapoport \cite{rapoport1948cycle}, Serrano and Bogu{\~n}a \cite{serrano2006clustering}, Watts and Strogatz \cite{watts1998collective} for more on this topic.) As a result, these graphs contain a large number of triangles, and many vertices in them lie in triangles. This is another motivation behind \eqref{eqn:407pm09sep21} and \eqref{eqn:415pm10sep21}. Large-deviation-type estimates of the form \eqref{eqn:407pm09sep21} and \eqref{eqn:415pm10sep21} are sometimes helpful to study the corresponding exponential random graph models. In dense settings, this connection has been investigated in Chatterjee and Diaconis \cite{chatterjee2013estimating}, Chatterjee and Dembo \cite{chatterjee2016nonlinear}, Bhamidi, Chakraborty, Cranmer and Desmarais \cite{bhamidi2018weighted}. Some potential implications of our results regarding exponential random graph models are discussed in Section \ref{sec:dis125pm19oct21}. In particular, we propose a new exponential random graph model, evaluate its normalising constant, and show that this model undergoes an intriguing phase transition. 

\medskip
\paragraph{\bf Organisation of this section.} Section~\ref{back} has provided us with the background and the motivation for the paper. Section~\ref{S2} introduces notation. In Section \ref{sec-many-triangles}, we estimate the probability for the Erd\H{o}s-R\'{e}nyi random graph to have many triangles. In Section \ref{sec-structure-many-triangles}, we study the structure of the graph conditionally on this event. In Section \ref{sec-many-vertices-in-triangles}, we estimate the probability for the Erd\H{o}s-R\'{e}nyi random graph to have many vertices in triangles. Sections \ref{sec-many-triangles}--\ref{sec-many-vertices-in-triangles} contain quick previews of the structure of the proofs of the main results presented there, which are written out in Sections \ref{S3}--\ref{S8}. We close in Section \ref{sec:dis125pm19oct21} with a discussion, formulating relevant extensions and open problems, addressing the relevance of our results for real-world networks, and giving a brief outline of the remainder of the paper.


\subsection{Introductory notation}
\label{S2}

A graph $G=(V,E)$ consists of a set of vertices $V=V(G)$ and a set of edges $E=E(G)$ connecting pairs of vertices. Abbreviate
$e_{\sss G} = |E(G)|$. 

In the complete graph on $n$ vertices, written $K_n$, all possible edges are present. In the Erd\H{o}s-R\'{e}nyi random graph on $n$ vertices with retention probability $p \in [0,1]$, written $G_{n,p}$, edges are present with probability $p$ and absent with probability $1-p$, independently for different edges. The law of $G_{n,p}$ is written $\mathbb{P}$. We denote the law of $G_{n,p}$ conditionally on the event that it contains a given graph $G$ with $V(G)\subseteq [n]$ by
	\eqn{
	\label{law-induced}
	\mathbb{P}_{\sss G}(\cdot) = \mathbb{P}(\,\cdot \mid G_{n,p} \supseteq G).
	}
Further, we let $\mathbb{E}_{\sss G}$ denote the (conditional) expectation w.r.t.\ $\mathbb{P}_{\sss G}$.

\begin{rmk}{\bf [Induced subgraph]}
\label{rem-induced}
	For a graph $G$ with $V(G)\subseteq [n]$, $G_{n,p} \supseteq G$ in \eqref{law-induced} simply means that the induced subgraph of $G_{n,p}$ on the vertex set $V(G)$ contains $G$ as a subgraph.\hfill$\blacksquare$
\end{rmk}
For a graph $G=(V,E)$, and $U\subseteq V$, $G[U]$ will denote the subgraph induced in $G$ by $U$. For two sequences of non-negative real numbers $(a_n)_{n\in\N}$ and $(b_n)_{n\in\N}$, $a_n=o(b_n)$ will mean $\lim_{n\rightarrow \infty}a_n/b_n=0$. For any two sequences $(a_n)_{n\in\N}$ and $(b_n)_{n\in\N}$, $a_n\approx b_n$ means $\lim_{n\rightarrow \infty}a_n/b_n=1$.


\subsection{Probability to have many triangles}
\label{sec-many-triangles}

Let $T$ be the number of triangles in $G_{n,p}$. In what follows, we fix $\lambda>0$, and put
	\eqn{
	\label{pn-vepn-def}
	p_n = \frac{\lambda}{n}, \qquad \varepsilon_n = k_n^{-2/3},
	}
where $(k_n)_{n\in\N}$ is a sequence of positive reals that can be chosen freely such that $\lim_{n\to\infty} k_n=\infty$. In order to state our main results on the large deviations for the number of triangles, we state the relevant {\em variational problem.} For $n\in\N$, $p \in (0,1)$ and $\constant,k>0$, define
\begin{equation}
\label{eqn:1054pm10mar21}
\Phi_{n,p,k}(\constant) := \min\{e_{\sss G}\log(1/p)\colon\, G \subseteq K_n,\,\mathbb{E}_{\sss G}(T)\geq \constant k\}.
\end{equation}
In words, $\mathrm{e}^{-\Phi_{n,p,k}(\constant)}$ is the probability that the expected number of triangles is at least $\constant k$ conditional on $G$ being present in $G_{n,p}$.

The following theorem provide asymptotically sharp upper and lower bounds on the probability of the event that $G_{n,p_n}$ contains at least $\constant k_n$ triangles:

\begin{thm}{\bf [Large deviations for the number of triangles]}
\label{thm-number-triangles}
Suppose that $\lim_{n\to\infty} \varepsilon_n = 0$. Fix a sequence of positive reals $(w_n)_{n\in\N}$ such that $\lim_{n\to\infty} w_n = 0$, and such that $w_nk_n \geq {n \choose 3} (p_n)^C$ for some constant $C>0$, $\lim_{n\to\infty} (\log{n})/(w_n^2 k_n^{1/3}) = 0$ and $\lim_{n\to\infty}(\log w_n)/\log(1/p_n)=0$. Then, for $n$ large enough,
	\eqn{
	\label{LD-many-triangles}
	- (1+\varepsilon_n)\Phi_{n,p_n,k_n}(\constant +w_n)\leq \log \mathbb{P}(T\geq \constant k_n) 
	\leq - (1-\psi_n)\Phi_{n,p_n,k_n}(\constant -w_n),
	}
for a sequence $\psi_n = o(1)$, where, for $\lim_{n\to\infty}$ $w_nk_n^{1/3}=\infty$,
	\eqn{
	\label{Phi-asymptotics}
	\tfrac12[6\constant(1-w_n)k_n]^{2/3} \log(1/p_n) \leq \Phi_{n,p_n,k_n}(\constant) 
	\leq \tfrac12[6\constant(1+w_n)k_n]^{2/3} \log (1/p_n).
	}
\end{thm}

\begin{rmk}{\bf [Structure of the proof of Theorem \ref{thm-number-triangles}]}
\label{rmk-proof-number-triangles}
The proof of Theorem \ref{thm-number-triangles} is given in three theorems stated separately below:\\ 
(a) The lower bound in \eqref{LD-many-triangles} in Theorem \ref{thm-number-triangles} is stated more generally in Theorem \ref{thm:437pm28apr21} in Section~\ref{S3}. In Theorem \ref{thm:437pm28apr21}, we estimate the probability $\mathbb{P}(T\geq \constant k_n)$ for a sequence of positive real numbers $k_n\to \infty$ as $n\to \infty$, and the estimate is expressed in terms of a perturbation of the variational formula given in \eqref{eqn:1054pm10mar21}. 
\\
(b) The upper bound in \eqref{LD-many-triangles} in Theorem \ref{thm-number-triangles} is stated more generally in Theorem \ref{thm:607pm22mar21} in Section \ref{S4}, and we refer to there for a more detailed discussion. For the upper bound we require the condition $\lim_{n\to\infty} (\log{n})/$ $(w_n^2 k_n^{1/3}) = 0$, which implies that we need $k_n$ to grow faster than $(\log n)^3$. On the other hand, the restriction on the perturbation term $w_n$ is somewhat more restrictive. We also need $\lim_{n\to\infty}(\log w_n)/$ $\log(1/p_n)=0$, which implies that $w_n$ can not converge to zero faster than $n^{-C}$ for any $C>0$. 
\\
(c) By monotonicity, the bounds imply the simpler statement that there exists a $\delta_n=o(1)$ such that
	\eqn{
	\label{LD-many-triangles-b}
	- (1+\delta_n)\Phi_{n,p_n,k_n}(\constant +\delta_n)\leq \log \mathbb{P}(T\geq \constant k_n) 
	\leq - (1-\delta_n)\Phi_{n,p_n,k_n}(\constant -\delta_n),
	}
which is obtained by taking $\delta_n=\min\{\varepsilon_n, w_n, \psi_n\}$. This weaker statement hides the dependencies between the various perturbations. Note that we can insert $\constant=1$ to obtain an even simpler estimate.
\\
(d) The asymptotics of the variational problem in \eqref{eqn:1054pm10mar21}, as described in \eqref{Phi-asymptotics}, is stated in Theorem \ref{thm:944pm28apr21} in Section \ref{S5}.\hfill$\blacksquare$
\end{rmk}

\begin{rmk}{\bf[Conditions of Theorem \ref{thm-number-triangles}]}
	The conditions in Theorem \ref{thm-number-triangles} ensure that the upper and lower bounds in \eqref{LD-many-triangles} are asymptotically equal up to the leading order term. For this we require that $\lim_{n \rightarrow \infty}\psi_n=0$, $\lim_{n \rightarrow \infty}\varepsilon_n=0$, and $\lim_{n \rightarrow \infty}w_n=0$. These give rise to the conditions in Theorem \ref{thm-number-triangles}. More precisely, \\
	(a) the term $\psi_n$ in the upper bound in \eqref{LD-many-triangles} converges to zero when $\lim_{n\to\infty} (\log{n})/(w_n^2 k_n^{1/3}) = 0$ and $\lim_{n\to\infty}(\log w_n)/\log(1/p_n)=0$. An explicit expression for $\psi_n$ is given in Theorem \ref{thm:607pm22mar21}. \\ 
	(b) the lower bound in \eqref{LD-many-triangles} holds even when $\varepsilon_n$ does not vanish as $n \to \infty$, but in order to get a lower bound that matches with the upper bound we need $\lim_{n \rightarrow \infty}\varepsilon_n = 0$. \\
	(c) since $T$ converges to a Poisson distribution with parameter $\tfrac{\lambda^3}{6}$ as $n\to \infty$, the upper bound in \eqref{LD-many-triangles} is not true when $\varepsilon_n$ does not vanish as $n \to \infty$.
\end{rmk}


\subsection{Graph structure conditionally on having many triangles}
\label{sec-structure-many-triangles}
We need the notion of a near-clique:

\begin{dfn}{\bf [Near-clique]}
\label{def-near-clique}
{\rm For $\delta>0$, $G$ is called an \emph{$(\delta,m)$-clique} if $G$ has minimum degree $(1-4\delta^{1/2})(2e_{\sss G})^{1/2}$, and $e_{\sss G}=m$.} \hfill$\spadesuit$ 
\end{dfn}

The following theorem shows that if we condition the random graph to have at least $\constant k_n$ triangles, then with high probability the graph contains an almost complete subgraph (see Section \ref{S6} for finer results), while at the same time it locally still looks like the unconditional Erd\H{o}s-R\'{e}nyi random graph.

\begin{thm}{\bf [Structure of the graph conditionally on having many triangles]}
\label{thm-structure-graph-many-triangles}
Under the conditions of Theorem \ref{thm-number-triangles}, for every $\delta>0$, there exists a constant $C>0$ and a sequence $\psi_n=o(1)$ such that, for $n$ large enough,
	\eqn{
	\label{existence-near-clique}
	\mathbb{P}\left(G_{n,p_n} \text{\rm contains no } \left(C(\psi_n+w_n),m_n\right)\text{\rm-clique} \mid T\geq \constant k_n\right) 
	\leq\delta, 
	}
where $m_n:= \tfrac12[6(\constant-2w_n)k_n(1-Ck_n^{-1/3})]^{2/3}$. At the same time, the Erd\H{o}s-R\'{e}nyi random graph, conditionally on $T\geq \constant k_n$, converges locally in probability to a $\text{Poi}(\lambda)$ branching process, just like the unconditional Erd\H{o}s-R\'{e}nyi random graph.
\end{thm}

\begin{rmk}{\bf [Structure of the proof of Theorem \ref{thm-structure-graph-many-triangles}]}
$\mbox{}$\\
(a) Theorem \ref{thm-structure-graph-many-triangles} shows that the triangles in the Erd\H{o}s-R\'{e}nyi random graph conditionally on having many triangles are highly localised: They are almost all in a near-clique, and thus the remainder of the graph is not affected by the conditioning.\\
(b) The proof of Theorem \ref{thm-structure-graph-many-triangles} is given in Section \ref{S6}, and is organised as follows. In Theorem \ref{thm:1000am24jun21} in Section \ref{sec-near-clique-many-triangles}, we prove a finer result on the existence of a near-clique. In Theorem \ref{thm:428pm11oct21} in Section \ref{sec-local-structure-many-triangles}, we identify the local limit of the Erd\H{o}s-R\'{e}nyi random graph conditionally on having many triangles. There we also define more precisely what local convergence in probability means. In this proof, we crucially rely on Bordenave and Caputo \cite[Theorem 1.8]{bordenave2015large}, which identifies the large deviation principle for the local limit of the Erd\H{o}s-R\'{e}nyi random graph. Both results are finer than the one stated in Theorem \ref{thm-structure-graph-many-triangles}.
\hfill$\blacksquare$
\end{rmk}


\subsection{Large deviations for the number of vertices in triangles}
\label{sec-many-vertices-in-triangles}

Let $V_{\sss T}(G)$ be the number of vertices in $G$ that are part of a triangle. The following theorem derives a large deviation estimate on the probability that the Erd\H{o}s-R\'enyi random graph has at least $k_n$ vertices that are part of a triangle:

\begin{thm}{\bf [Large deviations for the number of vertices in triangles]}
\label{thm-many-vertices-in-triangles}
Let $(k_n)_{n\in\N}$ be such that $\lim_{n \to \infty}k_n/\log n = \infty$. Then, for $n$ large enough,
	\eqn{
	-\tfrac13 k_n\log(\tfrac13k_n)-Ck_n\leq \log\mathbb{P}(V_{\sss T}(G_{n,p_n})\geq k_n) \leq  -\tfrac13 k_n\log(\tfrac13k_n)+Ck_n,
	}
for some constant $C>0$.
\end{thm}

\begin{rmk}{\bf [Structure of the proof of Theorem \ref{thm-many-vertices-in-triangles}]}
$\mbox{}$\\
(a) Comparing Theorem \ref{thm-many-vertices-in-triangles} to Theorem \ref{thm-number-triangles}, we see that the large deviations for the number of vertices in triangles and for the number of triangles are completely different: It is much more unlikely that there are many vertices in triangles than that there are many triangles. This is exemplified by Theorem \ref{thm-structure-graph-many-triangles}, which shows that the many triangles only involve few vertices.\\
(b) The proof of Theorem \ref{thm-many-vertices-in-triangles} is given in Section \ref{S7}, and is organised as follows. In Theorem \ref{thm:521pm10sep21} in Section \ref{S7.1}, we prove a finer upper bound on the number of vertices in triangles. Its proof is obtained by considering graphs in which all triangles are {\em disjoint}, and showing that the probability of such graphs is close to that of the upper bound. In Theorem \ref{thm:520pm10sep21} in Section \ref{S7.2}, we prove a finer upper bound on the probability of having many vertices in triangles. Its proof relies on a {\em decomposition} of the graph when there are many vertices in triangles, stated in Lemma \ref{lem:219pm23may21} in Section \ref{S7.2}, that is interesting in its own right. \hfill$\blacksquare$
\end{rmk}


\subsection{Discussion}
\label{sec:dis125pm19oct21}

In this section, we discuss our main results, mention some of their consequences, and list some open problems.

\medskip
\paragraph{\bf Large deviations for the number of triangles.} 
In Theorem \ref{thm-number-triangles} (see also Theorems \ref{thm:437pm28apr21} and \ref{thm:607pm22mar21}), we have obtained sharp upper and lower bounds on the probability $\log \mathbb{P}(T \geq \constant k_n)$ for a sequence of non-negative real numbers $(k_n)_{n\in \N}$. It is easy to see that, under the assumption $\lim_{n\rightarrow \infty}k_n^{-1/3} \log{n} =0$, Theorems \ref{thm:437pm28apr21} and \ref{thm:607pm22mar21} imply the weaker statement 
\begin{equation}
\label{eqn:315pm14sep21}
\lim_{n\rightarrow \infty}\frac{\log\mathbb{P}(T \geq k_n)}{k_n^{2/3} \log(1/p_n)}= - \tfrac12 6^{2/3}.
\end{equation}
Although the expression in \eqref{eqn:315pm14sep21} is much cleaner, the statements in Theorems \ref{thm:437pm28apr21} and \ref{thm:607pm22mar21} explicitly spell out the error terms. The proofs of these theorems are based on the theory of non-linear large deviations developed in Harel, Mousset, and Samotij \cite{harel2019upper} (which was partly based on Janson, Oleszkiewicz, and Ruci{\'n}ski \cite{janson2004upper}). 

\medskip
\paragraph{\bf Large deviations for the number of vertices in triangles.}
For $V_{\sss T}$, we can use Theorem \ref{thm-many-vertices-in-triangles} to write a limiting statement as well. For a sequence $k_n \to \infty$, such that $\lim_{n \to \infty}\frac{k_n}{\log n} = \infty$, and $k_n\leq n$,
\begin{equation}
	\lim_{n\rightarrow \infty}\frac{\log\mathbb{P}(V_{\sss T}(G_{n,p_n})\geq k_n) }{k_n\log(k_n)} = -\tfrac{1}{3}.
\end{equation}
Unfortunately, $V_{\sss T}$ (unlike $T$) does not appear to be a bounded-degree polynomial function of the entries of the adjacency matrix, and thus the existing methods are not readily applicable. More precisely, if $G$ is a simple graph on $n$ vertices with adjacency matrix $(a_{ij})_{i,j=1}^n$, then we can write 
\begin{equation}
V_{\sss T}(G)=\sum_{i=1}^n \left(1-\prod_{j=1}^n\prod_{k=1}^n(1-a_{ij}a_{ik}a_{jk})\right),
\end{equation} 
which is a polynomial function of the entries of the matrix $(a_{ij})_{i,j=1}^n$, and the degree of this polynomial depends on $n$, whereas the theory developed in Harel, Mousset, and Samotij \cite{harel2019upper} is only applicable to bounded-degree polynomials with non-negative coefficients. Also, it does not appear straightforward to apply the techniques in Chatterjee \cite{chatterjee2016nonlinear} (in particular, we were not able to verify the condition on $f$ in \cite[Theorem 1.1]{chatterjee2016nonlinear}). To get around this issue we have developed some novel {\em combinatorial} tools. Roughly, we have computed the number of graphs on $n$ vertices with $m$ edges such that exactly $q$ vertices are part of some triangle. In order to do so, we proved a {\em greedy decomposition lemma} (Lemma \ref{lem:219pm23may21} in Section \ref{S7.2}) for graphs whose $q$ vertices are part of some triangle and contain no `extra' edges. These techniques are a new contribution to the theory of non-linear large deviations, and we hope to develop them further in future work.

\medskip
\paragraph{\bf Extension: Local limit Erd\H{o}s-R\'{e}nyi conditioned on many vertices in triangles.}
It would be interesting to also investigate the {\em local limit} of the Erd\H{o}s-R\'{e}nyi random graph conditionally on having many vertices in triangles. It is not hard to show that this local limit does not change when $k_n$ is such that $k_n\log{n}=o(n)$. However, this rules out the interesting case where $k_n=\Theta(n),$ in which case this local limit should be {\em different}. We leave the proof of such a result to future work. Note, however, that it suffices to show that the local limit {\em exists}. Indeed, assume that the local limit exists along a subsequence, and that $V_{\sss T}(G_{n,p_n})\geq a n$ for some $a\in(0,1)$. Then, a uniform vertex will, with probability at least $a$, be part of a triangle. Since the Erd\H{o}s-R\'{e}nyi random graph is locally tree-like, this means that the limit is different. As such, a main ingredient would be to show that the Erd\H{o}s-R\'{e}nyi random graph conditionally on having at least $an$ vertices in triangles is {\em tight} in the local convergence topology.

\medskip
\paragraph{\bf Extension: Cliques for many triangles.}
In Theorem \ref{thm-structure-graph-many-triangles}, we show that, conditionally on there being many triangles, the graph whp contains a near-clique. We conjecture that this result can be considerably sharpened. Indeed, we expect that there is a clique of size $a_n$, where $a_n$ is the largest natural number such that $a_n(a_n-1)(a_n-2)/6\leq \constant k_n$. This is {\em much} sharper than the near-clique result in Theorem \ref{thm-structure-graph-many-triangles}.

\medskip
\paragraph{\bf Extension: Beyond triangles.}
We expect that some of our results can be relatively straightforwardly extended beyond triangles. Indeed, Theorem \ref{thm-number-triangles} should apply also when dealing with larger cliques, when the variational problem in \eqref{eqn:1054pm10mar21} is appropriately adapted. We expect that this is not quite the case for Theorem \ref{thm-many-vertices-in-triangles}. The proof of Theorem \ref{thm-many-vertices-in-triangles} crucially relies on a counting argument based on the decomposition Lemma \ref{lem:219pm23may21} in Section \ref{S7.2}. Although this lemma extends to higher-order cliques, the counting argument becomes much more difficult.

\medskip
\paragraph{\bf Extension: Other ranges of $p_n$.} We believe that Theorem \ref{thm-number-triangles} can be extended to larger values of $p_n$ by choosing a suitable sequence $k_n$, but Theorem \ref{thm-many-vertices-in-triangles} does not seem to extend easily. Indeed, in our proof (see Lemma \ref{lem:210pm23jun21} in Section \ref{S7.2}) we have used the fact that the probability for the graph to be triangle free is bounded below by a constant, and the proof of Lemma \ref{lem:540pm24may21} in Section \ref{S7.2} also suggests that if we let $\lambda$ grow like a power of $n$, then it will affect the upper bound in Theorem \ref{thm-many-vertices-in-triangles}. It is still possible to let $\lambda$ grow slowly, say, like $\log{n}$, but we do not pursue this in this paper.   

\medskip
\paragraph{\bf Highly-clustered real-world networks.}
In many real-world networks there are order $n$ triangles, as well as order $n$ vertices in triangles. Thus, from a practical perspective, the setting where $k_n$ is of order $n$ is the most relevant. Comparing Theorems \ref{thm-number-triangles} and \ref{thm-many-vertices-in-triangles}, we see that conditioning on many vertices in triangles is a {\em much more effective way to create a highly-clustered graph} than conditioning on having many triangles. Furthermore, Theorem \ref{thm-many-vertices-in-triangles} also explains that the logarithm of the probability for an Erd\H{o}s-R\'{e}nyi random graph to be highly clustered is of the order $n\log{n}$, a regime that has not yet attracted much attention. Indeed, Bordenave and Caputo \cite{bordenave2015large} show that the probability for the local limit of the Erd\H{o}s-R\'{e}nyi random graph to be unequal to a Poisson branching process is exponentially small, but {\em only when} the alternative local limit is a tree itself. It would be highly interesting to investigate the probability that the local limit is not a tree in more detail, and the exponential rate should be $n\log{n}$.

\medskip
\paragraph{\bf Exponential random graphs as real-world network models.}
Exponential random graphs models (ERGM) are popular for modelling real-world networks. Let $\mathcal{G}_n$ be the space of all simple graphs on the vertex set $[n]$, which has $2^{n\choose 2}$ elements. The ERGM can be represented by the following exponential form
	\begin{equation}
	\mathbb{P}_{\sss T}(G)=\frac{1}{Z_n}\exp{\left(T(G)\right)},\,\,G\in \mathcal{G}_n,
	\end{equation}	   
where $T(G)$ is a function on the space of graphs, and $Z_n$ is the normalizing constant (also called the partition function). Examples of $T$ include subgraph counts such as the number of edges, triangles, cycles, etc. Such models were first studied in Holland and Leinhardt \cite{holland1981exponential}, Frank and Strauss \cite{frank1986markov}. Several new sufficient statistics were introduced in Snijders, Pattison, Robins and Handcock \cite{snijders2006new}. The evaluation of $Z_n$ is one of the fundamental (and often difficult) problems. In the dense regime the first such result was obtained by Chatterjee and Diaconis \cite{chatterjee2013estimating}. While ERGMs are well-understood in the dense regime, there are hardly any results on sparse random graphs (a recent result in the sparse regime appeared in Mukherjee \cite{mukherjee2020degeneracy}, and a related model called the random triangle model was studied by Jonasson \cite{jonasson1999random}, and H{\"a}ggstr{\"o}m and Jonasson \cite{haggstrom1999phase}). Unfortunately, even dense exponential random graphs are problematic, as shown by Bhamidi, Bresler, and Sly \cite{BhaBreSly11}: either they locally look like dense Erd\H{o}s-R\'{e}nyi random graphs, or the mixing times of Glauber dynamics or Metropolis-Hasting dynamics on these graphs are exponentially large. 

\medskip
\paragraph{\bf Phase transitions for exponential random graphs.}
Our results have an interesting tale to tell about exponential random graphs. First, we could consider exponential random graphs given by the measure $\mathbb{P}_\beta$ defined in terms of its Radon-Nikodym derivative w.r.t.\ the law $\mathbb{P}$ of the Erd\H{o}s-R\'{e}nyi random graph with $p=\lambda/n$ as
	\eqn{
	\label{ERRG-vertices}
	\frac{d\mathbb{P}_\beta}{d\mathbb{P}}=\frac{1}{Z_n(\betanV)} \mathrm{e}^{\betanV V_{\sss T}},
	}
where $Z_n(\betanV)$ is a normalizing constant or {\em partition function} and $\betanV\geq 0$ is an appropriate sequence. Theorem \ref{thm-many-vertices-in-triangles} suggests that this measure is well-defined for $\betanV=\beta\log{n}$, and that it will indeed give rise to a linear number of vertices in triangles. Note that the $\beta \log{n}$ scaling in the exponential is crucial to make the measure well-defined in the large-graph limit. ERGMs often exhibit {\em phase transitions} (also called degeneracy) at some values of the parameters.(see Chatterjee and Diaconis \cite{chatterjee2013estimating}, Bhamidi, Chakraborty, Cranmer, and Desmarais \cite{bhamidi2018weighted}, Mukherjee \cite{mukherjee2020degeneracy}). Roughly, this means that a slight change in the values of some parameter has a drastic effect on a typical outcome of the model.

The above phenomenon is observed in various models in statistical physics (see Georgii \cite{georgii2011gibbs}), and is intimately related to the properties of the associated partition function. Here we investigate the scaling of the partition function, and show that the model proposed in \eqref{ERRG-vertices} undergoes a phase transition:

\begin{cor}{\bf [Exponential graph with the number of vertices in triangles]}
\label{cor-ERRG-vertices-in-triangles}
$\mbox{}$\\
Consider the exponential random graph model defined in \eqref{ERRG-vertices} with $\betanV=\beta\log{n}$. Then
	\eqn{
	\lim_{n\rightarrow \infty} \frac{1}{n\log{n}} \log{Z_n(\beta\log n)}=(\beta-\tfrac{1}{3})_+,
	}
where $(x)_+=\max\{x,0\}$. Moreover, when $\beta<\tfrac13$, $V_{\sss T}/n$ converges in probability to $0$, and when $\beta>\tfrac13$, $V_{\sss T}/n$ converges in probability to $1$ as $n\to\infty$.
\end{cor}

Note that Corollary \ref{cor-ERRG-vertices-in-triangles} not only determines the scaling of the partition function, it also shows that the model undergoes a phase transition: when $\beta<\tfrac13$, in a typical realisation from $\mathbb{P}_{\beta\log{n}}$ very few vertices are part of some triangle, while when $\beta>\tfrac13$, almost all vertices are part of some triangle. The proof of Corollary \ref{cor-ERRG-vertices-in-triangles} is given in Section \ref{S8}. Finally, Corollary \ref{cor-ERRG-vertices-in-triangles} and Theorem \ref{thm-many-vertices-in-triangles} suggest that interesting behaviour arises for $\betanV=\frac{1}{3}\log{n}+\beta$. For this we would need to study $\log\mathbb{P}(V_{\sss T}=k_n)$ up to order $n$, which is beyond the scope of the present paper.

Alternatively, we could consider the measure
	\eqn{
	\label{ERRG-triangles}
	\frac{d\mathbb{P}_{\betanT, \betanS}}{d\mathbb{P}}=\frac{1}{Z_n(\betanT, \betanS)} \mathrm{e}^{\betanT  T -\betanS S},
	}
where $S$ is the number of subgraphs $K_4\setminus e$, the complete subgraph on $4$ vertices minus an edge. These subgraphs are sometimes called {\em diamonds}. Again, we take $\betanT=\beta_{\sss T}\log{n}$ and $\betanS=\beta_{\sss S}\log{n}$ with $\beta_{\sss T},\beta_{\sss S}\geq 0$, as suggested by Theorem \ref{thm-many-vertices-in-triangles}, when we wish to have order $n$ vertices in triangles. When $\betanS=0$, Theorem \ref{thm-number-triangles} suggests that the graph will be complete whp, as this gives a contribution $n(n-1)(n-2)/6$ in the exponent, which is so large that it cannot be beaten by the entropy of graphs that miss a substantial number of edges. However, when $\betanS=\beta_{\sss S}\log{n}$ with $\beta_{\sss S}>0$, the $-\betanS S$ term strongly penalises graph configurations with large cliques, and thus `tames' the triangle count. As a result, we are tempted to believe that the measure $\mathbb{P}_{\betanT, \betanS}$ is concentrated on graph configurations with a linear number of triangles in a linear number of vertices. This would lead to sparse exponential random graphs that are {\em not locally tree-like}. It would be of great interest to see whether such results can be proved. They would be extremely relevant for the study of sparse exponential random graphs. Unfortunately, we do not know how to prove a result like Corollary \ref{cor-ERRG-vertices-in-triangles} for $\betanT=\beta_{\sss T}\log{n}$ and $\betanS=\beta_{\sss S}\log{n}$ in this model.

\medskip
\paragraph{\bf Outline.}
Sections~\ref{S3}--\ref{S8} provide the proofs of our main results. Sections~\ref{S3}--\ref{S5} focus on Theorem~\ref{thm-number-triangles}, Section~\ref{S6} on Theorem~\ref{thm-structure-graph-many-triangles}, Section~\ref{S7} on Theorem~\ref{thm-many-vertices-in-triangles}, and Section \ref{S8} on Corollary~\ref{cor-ERRG-vertices-in-triangles}. Along the way, we define refined upper and lower bounds that are of interest in themselves.


\section{Lower bound on the large deviations for the number of triangles}
\label{S3}

In this section, we prove the lower bound in Theorem \ref{thm-number-triangles}. The following theorem describes a more precise result:

\begin{thm}{\bf [Large deviation lower bound]}
\label{thm:437pm28apr21}
Suppose that the sequence of positive reals $(\varepsilon_n)_{n\in\N}$ is such that $\lim_{n\to\infty} \varepsilon_n = 0$ and $\liminf_{n\rightarrow \infty}k_n^{2/3}\varepsilon_n>0$. Then, for every sequence of positive reals $(w_n)_{n\in\N}$ such that $\lim_{n\to\infty}$ $ w_n = 0$, and such that $w_nk_n \geq {n \choose 3} (p_n)^C$ for some constant $C>0$, for $n$ large enough,
\begin{equation}
\log \mathbb{P}(T\geq \constant k_n) \geq -(1+\varepsilon_n)\Phi_{n,p_n,k_n}(\constant +w_n).
\end{equation}
\end{thm}

\begin{rmk}{\bf [Parameters in Theorem \ref{thm:437pm28apr21}]}
In Theorem \ref{thm:437pm28apr21} we estimate the probability $\mathbb{P}(T\geq \constant k_n)$ for a sequence of positive real numbers $(k_n)_{n\in\N}$ such that $\lim_{n\to\infty} k_n=\infty$, and the estimate is expressed in terms of a perturbation of the variational formula in \eqref{eqn:1054pm10mar21}. As noted before, the condition $\lim_{n\to\infty} \varepsilon_n = 0$ requires that $k_n$ grows faster than $(\log{n})^{3/2}$. The variational formula is perturbed by $w_n$, and the condition $w_nk_n \geq {n \choose 3} (p_n)^C$ requires that the perturbation is not too small. More precisely, this condition implies that $w_n\geq \frac{1}{n^Ck_n}$ for some constant $C>0$ as $n\to\infty$. Also note that we can easily estimate $\mathbb{P}(T\geq k_n)$ by simply picking $\constant=1$. We have chosen to express the theorem in the current form because it makes the perturbation argument in our proof notationally cleaner. \hfill$\blacksquare$
\end{rmk}

The proof of Theorem \ref{thm:437pm28apr21} is based on the idea that $\mathbb{P}(T\geq \constant k_n)$ is controlled by the presence of a subgraph $G$ that minimises an appropriate perturbation of  \eqref{eqn:1054pm10mar21}. This subgraph is very dense, as we prove later on. In the rest of this section we formalise this idea. 

Let $G$ be the graph that attains the minimum in the definition of $\Phi_{n,p_n,k_n}(\constant +w_n)$ in \eqref{eqn:1054pm10mar21}. First note that, by the definition of conditional probability,
	\begin{equation}
	\label{aln:1119am03oct21}
	\begin{aligned}
		-\log\mathbb{P}(T\geq \constant k_n) 
		&\leq - \log \mathbb{P}(G_{n,p_n} \supseteq G) - \log \mathbb{P}_{\sss G}(T\geq \constant k_n)\\
		& = e_{\sss G} \log(1/p_n) - \log \mathbb{P}_{\sss G}(T\geq \constant k_n).
	\end{aligned} 
\end{equation}
Therefore, by construction, we have $e_{\sss G} \log(1/p_n)=\Phi_{n,p_n,k_n}(\constant +w_n)$. We will estimate $\mathbb{P}_{\sss G}(T\geq \constant k_n)$ in \eqref{aln:1119am03oct21}, for which we use the following lemma, whose proof is postponed to the end of this section.

\begin{lem}{\bf [Lower bound on the edge count for graphs with many triangles]}
\label{lem:1109pm24mar21}
Let $G$ be any simple graph on $n$ vertices with adjacency matrix $(a_{ij})_{i,j=1}^n$. Let $e_{\sss G}$ and $\TG$ be the number of edges and triangles in $G$, respectively. Then $\frac{1}{6}(2e_{\sss G})^{3/2} \geq \TG$.
\end{lem}

We use the standard notation where $K_r$ denotes the complete graph on $r$ vertices, and $K_{r,s}$ denotes the complete bipartite graph with $r$ vertices in one part and $s$ vertices in the other part. Now, using the fact that $G$ is the graph that attains the minimum in the definition of $\Phi_{n,p_n,k_n}(\constant +w_n)$ (see \eqref{eqn:1054pm10mar21}), we get 	
	\begin{equation}
	\constant k_n\leq (\constant +w_n)k_n \leq \mathbb{E}_{\sss G}(T) \leq N(K_{1,2},G)\,p_n+N(K_2,G)\,np_n^2+N(K_3,G)+\tfrac16\lambda^3,
	\end{equation}
where $N(H,G)$ is the number of copies of $H$ in $G$, and we recall \eqref{pn-vepn-def}. Note that $N(K_{1,2},G) \leq N(K_2,G) |V(G)|$ (since each edge in $G$ can be in at most $|V(G)|$ many $K_{1,2}$'s). Hence, for some constant $C>0$,
	\begin{equation}
	\begin{aligned}
	\constant k_n &\leq \mathbb{E}_{\sss G}(T) \leq  N(K_2,G) |V(G)|\,p_n+N(K_2,G)\,np_n^2+N(K_3,G)+\tfrac16\lambda^3\\
	&\leq \left(\lambda+ \tfrac{1}{n}\lambda^2\right)N(K_2,G) + N(K_3,G) + \tfrac16 \lambda^3 
	\leq C\left(N(K_2,G)\right)^{3/2},
	\end{aligned}
	\end{equation}
where in the third inequality we use Lemma \ref{lem:1109pm24mar21}. This gives $e_{\sss G} \geq Ck_n^{2/3}$ for some constant $C$. Again using the fact that $G$ is the graph that attains the minimum in the definition of $\Phi_{n,p_n,k_n}(\constant +w_n)$, we get
	\begin{equation}
	\label{eqn:eqn639pm2mar21}
	(\constant +w_n)k_n \leq \mathbb{E}_{\sss G}(T) \leq {n \choose 3} \mathbb{P}_{\sss G}(T\geq \constant k_n)+\constant k_n,
	\end{equation}
and so
	\begin{equation}
	\label{eqn:1109pm10mar21}
	\mathbb{P}_{\sss G}(T\geq \constant k_n) \geq w_n k_n/{n \choose 3}.
	\end{equation}
Therefore, substituting \eqref{eqn:1109pm10mar21} into \eqref{aln:1119am03oct21}, we get
	\begin{equation}
	\label{aln:1119pm10mar21}
	\begin{aligned}
	-\log\mathbb{P}(T\geq \constant k_n) \leq e_{\sss G} \log(1/p_n) + C \log(1/p_n)
	\end{aligned}
	\end{equation}
for some constant $C$. Here we use our assumption that ${n\choose 3}/w_nk_n \leq \left(1/p_n\right)^C$ for some constant $C>0$. By construction, $\Phi_{n,p_n,k_n}(\constant +w_n) = e_{\sss G} \log(1/p_n)$, and we have already observed that $e_{\sss G} \geq Ck_n^{2/3}$. Therefore
	\begin{equation}
	\label{eqn:1127pm10mar21}
	-\log\mathbb{P}(T\geq \constant k_n) \leq (1+\varepsilon_n)\Phi_{n,p_n,k_n}(\constant +w_n)
	\end{equation}
when $\lim_{n\to\infty} \varepsilon_n = 0$ and $\liminf_{n\rightarrow \infty} k_n^{2/3}\varepsilon_n>0.$ This completes the proof of Theorem \ref{thm:437pm28apr21}. We conclude this section with the proof of Lemma \ref{lem:1109pm24mar21}.

\begin{proof}[Proof of Lemma \ref{lem:1109pm24mar21}]
	By Cauchy-Schwarz,
	\begin{equation}
	\begin{aligned}
		6\TG &=\sum_{i,j,k=1}^n a_{ij}a_{jk}a_{ki} =\sum_{i,j=1}^n a_{ij}\sum_{k=1}^n a_{jk}a_{ki}\\
		&\leq \Big(	\sum_{i,j=1}^n a_{ij}\Big)^{1/2} 
		\Big(\sum_{i,j=1}^n \Big(\sum_{k=1}^na_{jk}a_{ki}\Big)^2\Big)^{1/2} \\
		&\leq \Big(	\sum_{i,j=1}^n a_{ij}\Big)^{1/2} \Big(\sum_{i,j=1}^n \sum_{k=1}^na_{jk} 
		\sum_{l=1}^n a_{li}\Big)^{1/2} = \left(2e_{\sss G}\right)^{3/2}.
	\end{aligned}
	\end{equation}
\end{proof}


\section{Upper bound on the large deviations for the number of triangles}
\label{S4}

In this section, we prove the upper bound in Theorem \ref{thm-number-triangles}. The following theorem describes a more precise result:

\begin{thm}{\bf [Large deviation upper bound]}
\label{thm:607pm22mar21}
For every $\delta>0$, and every sequence of positive reals $(w_n)_{n\in\N}$ such that $\lim_{n\to\infty}$ $ w_n = 0$, $\lim_{n\to\infty} (\log{n})/(w_n^2 k_n^{1/3}) = 0$ and $\lim_{n\to\infty}(\log w_n)/$ $\log(1/p_n)=0$,
\begin{equation}
\log \mathbb{P}(T \geq \constant k_n) \leq -(1-\psin)\Phi_{n,p_n,k_n}(\constant-2w_n) + \log(1+\delta)
\end{equation}
for $n$ large enough, where 
\begin{equation}
\label{psindefextra}
\psin = \frac{C' \log{n}}{w_n^2k_n^{1/3}}+\frac{\log[C'w_n^{-5}\log^3(1/p_n)]}{\log(1/p_n)}
\end{equation}
for some constant $C'>0$.
\end{thm}

\begin{rmk}{\bf [Parameters in Theorem \ref{thm:607pm22mar21}]}
In Theorem \ref{thm:607pm22mar21} we require $k_n$ to grow faster than $(\log{n})^{3}$ as $n\to \infty$. On the other hand, the restriction on the perturbation term $w_n$ is somewhat more restrictive. More precisely, we need $w_nk_n^{1/6} \gg (\log n)^{1/2}$. We also need $(\log w_n)/\log(1/p_n)\rightarrow 0$, which implies that $w_n$ cannot converge to zero faster than $n^{-C}$ for any $C>0$. Again, in Theorem \ref{thm:607pm22mar21}, we can pick $\constant=1$ to obtain an estimate on $\mathbb{P}(T \geq k_n)$. \hfill$\blacksquare$
\end{rmk}

The remainder of this section is devoted to the proof of Theorem \ref{thm:607pm22mar21}. This is done in two steps. In the first step, in Section~\ref{S4.1}, we show that if a large number of triangles are present in $G_{n,p_n}$, then the graph contains a particular subgraph called a `seed' (Definition \ref{def-seed} and Lemma \ref{lem:117pm15mar21}), which contains a specific subgraph called a `core' (Definition \ref{def-core} and Lemma \ref{lem:317pm22mar21}). 

We next obtain an upper bound on the number of cores (Lemma \ref{lem:307pm22mar21}). In the second step, in Section ~\ref{S4.2}, we upper bound the probability of $G_{n,p_n}$ having at least $\constant k_n$ many triangles by the probability of $G_{n,p_n}$ having the appropriate number of cores, and then use the estimates obtained in the first step to obtain a sharp upper bound on this probability.


\subsection{Triangles are created by seeds}
\label{S4.1}

In this section, we first show that if $G_{n,p_n}$ contains a large number of triangles, then it must contain a particular subgraph called `seed' with `appropriately high' probability. 

\begin{dfn}{\bf [Seed]}
\label{def-seed}
{\rm For $C$ sufficiently large and for a sequence of positive reals $(w_n)_{n\in\N}$ such that $\lim_{n\to\infty} w_n =0$, we call $G$ a \emph{seed} if the following hold:
\begin{enumerate}
\item[\rm (S1)] $\mathbb{E}_{\sss G}(T) \geq (\constant -w_n)k_n$.
\item[\rm (S2)] $e_{\sss G} \leq {C\constant}{w_n^{-1}}k_n^{2/3}\log(1/p_n)$. \hfill $\spadesuit$
\end{enumerate}
} 
\end{dfn}
\noindent
A seed is defined for all $n$, although later we only need it for $n$ sufficiently large.

The following lemma reduces the study of triangles to the existence of a seed. We obtain a quantitative estimate on the probability that $G_{n,p_n}$ contains a large number of triangles, but does not contain a seed of appropriate size. The proof is an adaptation of \cite{harel2019upper}, which in turn was based on a classical moment argument in \cite{janson2004upper}.

\begin{lem}{\bf [Unlikely not to contain a seed]}
\label{lem:117pm15mar21}
Let $(w_n)_{n\in\N}$ be such that $\lim_{n\to\infty} w_n =0$. Then for every positive integer $\ell\leq \frac{C\constant}{3}{w_n^{-1}}k_n^{2/3}\log(1/p_n)$, and every $\constant>0$,
\begin{equation}
\label{eqn:125pm06mar21}
\mathbb{P}\left(T\geq \constant k_n,\,G_{n,p_n} \text{\rm contains no seed}\right) \leq \left(1-\tfrac{w_n}{\constant}\right)^{\ell}
\end{equation}
for any $(k_n)_{n\in\N}$ such that $\lim_{n\to\infty} k_n = \infty$. Consequently,
	\begin{equation}
	\label{eqn:243pm04oct21}
	\mathbb{P}(T \geq \constant k_n) \leq (1+\xi_n)\,\mathbb{P}(G_{n,p_n} \text{\rm contains a seed}),
	\end{equation}
and
	\begin{equation}
	\label{eqn:242pm04oct21}
	\mathbb{P}\left(T\geq \constant k_n,\,G_{n,p_n} \text{\rm contains no seed}\right) \leq \xi_n \mathbb{P}\left(T\geq \constant k_n\right),
	\end{equation}
where 
	\begin{equation}
	\label{eqn:xin}
	\xi_n = \exp\left([\tfrac12 (6\constant)^{2/3}-\tfrac13 C]\,k_n^{2/3}\log(1/p_n)\right).
	\end{equation}
\end{lem}

\begin{proof} 
We use a bound on the moments of the number of triangles, on the event that $G_{n,p_n}$ does not contain a seed. For this, let $Z$ be the indicator of the event that $G_{n,p_n}$ does not contain a seed. Fix a positive integer $\ell$ that satisfies $\ell\leq \frac13 C\constant w_n^{-1}k_n^{2/3}\log(1/p_n) $. Then $TZ\geq 0$ and $Z^{\ell} = Z$, so that, by Markov's inequality,
	\begin{equation}
	\label{eqn:944ammar21}
	\mathbb{P}\left(T\geq \constant k_n,\,G_{n,p_n} \text{ contains no seed}\right) 
	= \mathbb{P}\left(TZ\geq \constant k_n\right) \leq \frac{\mathbb{E}\left(T^{\ell}Z\right)}{(\constant k_n)^{\ell}}.
	\end{equation}
Also note that we can write $T=\sum_t Y_t$, where the sum runs over all triangles in $K_n$, and $Y_t$ is the indicator that $t$ belongs to $G_{n,p_n}$. 

Now let $Z_{\sss G}$ be the indicator of the event that $G \cap G_{n,p_n}$ does not contain a seed. This gives $Z_{\sss G} \leq Z_{\sss G'}$ when $G' \subseteq G$. Also, since $Z=Z_{\sss K_n}$ for $k \in  [\ell]$, we get
	\begin{equation}
	\begin{aligned}
	\mathbb{E}(T^kZ)
	&= \sum_{t_1,\ldots,t_k}\mathbb{E}(Y_{t_1}Y_{t_2}\cdots Y_{t_k}Z) \\
	&\leq \sum_{t_1,\ldots,t_k}\mathbb{E}(Y_{t_1}Y_{t_2}\cdots Y_{t_k}Z_{t_1\cup\cdots  \cup t_k}) \\
	& \leq \sum_{t_1,\ldots,t_{k-1}}\mathbb{E}(Y_{t_1}Y_{t_2}\cdots Y_{t_{k-1}}Z_{t_1\cup\cdots  \cup t_{k-1}} ) 
	\sum_{t_k} \mathbb{E}(Y_{t_{k}} \mid Y_{t_1}Y_{t_2}\cdots Y_{t_{k-1}}Z_{t_1\cup\cdots  \cup t_{k-1}}=1).
	\end{aligned}
	\end{equation}	
Note that the first sum runs over all $t_1,\cdots, t_{k-1}$ such that the event $Y_{t_1}Y_{t_2}\cdots Y_{t_{k-1}}Z_{t_1\cup\cdots  \cup t_{k-1}}=1$ has positive probability. This implies that $t_1\cup\cdots  \cup t_{k-1}$ (seen as a graph containing the edges and vertices in $t_i$ for $i\in[k-1]$) does not contain a seed. Also note that $e_{t_1\cup\cdots  \cup t_{k-1}} \leq 3(k-1) \leq 3\ell \leq C\constant w_n^{-1}k_n^{2/3}\log(1/p_n)$, and so
	\begin{equation}
	\sum_{t_k}\mathbb{E}( Y_{t_{k}} \mid Y_{t_1}Y_{t_2}\cdots Y_{t_{k-1}}=1)
	= \mathbb{E}_{t_1\cup\cdots  \cup t_{k-1}}(T) <(\constant -w_n)k_n,
	\end{equation}
because otherwise $t_1\cup\cdots  \cup t_{k-1}$ would be a seed. Therefore
	\begin{equation}
	\sum_{t_1,\ldots,t_k}\mathbb{E}(Y_{t_1}Y_{t_2}\cdots Y_{t_k}Z_{t_1\cup\cdots  \cup t_k})
	< (\constant -w_n)k_n \sum_{t_1,\ldots,t_{k-1}}\mathbb{E}(Y_{t_1}Y_{t_2}\cdots Y_{t_{k-1}}Z_{t_1\cup\cdots  \cup t_{k-1}}).
	\end{equation}
Repeating this argument, we obtain $\mathbb{E}(T^{\ell} Z) \leq (\constant -w_n)^{\ell} k_n^{\ell}$. Therefore, using \eqref{eqn:944ammar21}, we can write
	\begin{equation}
	\label{eqn:959ammar21}
	\mathbb{P}\left(T\geq \constant k_n,\, G_{n,p_n} \text{ contains no seed}\right)  
	\leq \left(1-\tfrac{w_n}{\constant}\right)^{\ell}.
	\end{equation}
Finally, putting $\ell = \frac13 C\constant w_n^{-1}k_n^{2/3}\log(1/p_n)$, we get
	\begin{equation}
	\left(\tfrac{\constant-w_n}{\constant}\right)^{\tfrac13 C\constant w_n^{-1}k_n^{2/3}\log(1/p_n)} 
	= \exp\left(\tfrac13 C\constant w_n^{-1}k_n^{2/3}\log(1/p_n) \log{\left(\tfrac{\constant-w_n}{\constant}\right)}\right).
	\end{equation}
Since $\log(1-x) \leq -x$, $x \in (0,1)$, we arrive at
	\begin{equation}
	\label{eqn:204pm25mar21}
	\left(\tfrac{\constant-w_n}{\constant}\right)^{\ell} \leq  \exp\left(-\tfrac13 Ck_n^{2/3}\log(1/p_n) \right) 
	\leq \xi_n\,\mathbb{P}\left(G_{n,p_n} \text{ contains a clique of size }(6\constant k_n)^{1/3}\right),
	\end{equation}
where $\xi_n$ is given in \eqref{eqn:xin}. Note that a clique of size $(6\constant k_n)^{1/3}$ is a seed because $\binom{(6ak_n)^{1/3}}{2} \leq k_n^{2/3}\log (1/p_n)/w_n$ when $w_n=o(1)$, and such a clique contains $\constant k_n$ many triangles. This concludes the proofs of \eqref{eqn:242pm04oct21} and \eqref{eqn:243pm04oct21}.
\end{proof}

It turns out that the existence of a seed implies existence of a minimal type of seed that we call a {\em core}:

\begin{dfn}{\bf [Core]} 
\label{def-core}
{\rm A graph $G^*$  is called a {\em core} if the following conditions hold:
\begin{enumerate}
\item[(C1)] $\mathbb{E}_{G^*}(T)\geq (\constant-2w_n)k_n$.
\item[(C2)] $e_{\sss G^*} \leq {C\constant}{w_n^{-1}}k_n^{2/3}\log(1/p_n)$.
\item[(C3)] $\min_{e \in E(G^*)}(\mathbb{E}_{G^*}(T)-\mathbb{E}_{G^*\setminus e}(T)) \geq w_n^2 \frac{k_n^{1/3}}{C\constant \log(1/p_n)}$.
\end{enumerate}
Moreover, we call a core $G^*$ an {\em $m$-core} when it has exactly $m$-edges. \hfill $\spadesuit$}
\end{dfn}

Intuitively, a core can be thought of as a {\em minimal} version of a seed. We will shortly see that obtaining an upper bound on the number of cores is a relatively accessible task. First, let us show that every seed $G$ contains a core $G^*$:

\begin{lem}{\bf [Seeds contain cores]}
\label{lem:317pm22mar21}
Every seed $G$ contains a core $G^*$.
\end{lem}

\begin{proof}
We start with a seed $G$, and recursively remove edges to create a core. For this, we define a sequence of graphs $G=G_0\supseteq G_1 \supset \cdots \supseteq G_s =G^*$. Here, $G_{i+1} = G_i\setminus e$ for some edge in $e \in G_i$ such that $\mathbb{E}_{\sss G_i}(T)-\mathbb{E}_{\sss G_i\setminus e}(T)< w_n^2 \frac{k_n^{1/3}}{C\constant \log(1/p_n)}$. We continue deleting edges as long as there is such an $e$. We stop when such an $e$ no longer exists, and denote the resulting graph by $G^*$. Therefore, $G^*$ satisfies Definition \ref{def-core} (C3) and (C2) by construction. Also note that the number $s$ of edges removed in the above procedure satisfies $s \leq e_{\sss G} \leq  {C\constant}{w_n^{-1}}k_n^{2/3}\log(1/p_n)$ since $G$ is a seed, and therefore
	\begin{equation}
	\label{eqn:115pm16mar21}
	\mathbb{E}_{\sss G}(T)-\mathbb{E}_{\sss G^*}(T) 
	= \sum_{i=0}^{s-1} \left(\mathbb{E}_{\sss G_{i}}(T)-\mathbb{E}_{\sss G_{i+1}}(T)\right)<s w_n^2
	\frac{k_n^{1/3}}{C\constant \log(1/p_n)}\leq w_nk_n.
	\end{equation}
Now, since $G$ is a seed, we have $\mathbb{E}_{\sss G}(T) \geq (\constant -w_n)k_n$. Combining this with \eqref{eqn:115pm16mar21}, we get $\mathbb{E}_{\sss G^*}(T)\geq (\constant-2w_n)k_n$, which is Definition \ref{def-core}(C1).
\end{proof}

In the following lemma we obtain an upper bound on the number of $m$-cores that plays a crucial role in the proof of Theorem \ref{thm:607pm22mar21}:

\begin{lem}{\bf [Bound on the number of cores]}
\label{lem:307pm22mar21}
Let $\lim_{n\rightarrow \infty}\frac{\log{n}}{w_n^2k_n^{1/3}}=0$. Then the number of $m$-cores is bounded above by $(1/p_n)^{m \Psi_n}$, where
	\begin{equation}
	\label{eqn:436pm21mar21}
	\Psi_n= \frac{1}{\log(1/p_n)} \left(\tfrac{C'}{t_n} \log {n} +\log\left(C'\tfrac{m}{t_n^2}\right)\right) \quad
	\text{  and  } \quad t_n := w_n^2 \frac{k_n^{1/3}}{C\constant \log(1/p_n)}.
	\end{equation}
\end{lem}

\begin{proof}
For an $m$-core $G^*$, let $A_{\sss G^*} \subseteq V(G^*)$ be the set of vertices with degree at least $t_n -C(\lambda)$, where $C(\lambda)=\lambda+\lambda^2$. Since $G^*$ has $m$ edges, 
	\begin{equation}
	\label{eqn:1145am17mar21}
	\left|A_{\sss G^*}\right|\leq 2m\left(t_n-C(\lambda)\right)^{-1} :=\ellalt.
	\end{equation}
We will show that the endpoints of every edge in $G^*$ lie in $A_{\sss G^*}$. (Note that $A_{\sss G^*} $ does not have to equal $V(G^*)$, because the addition of an isolated vertex need not destroy the core property.) To see why, fix an edge $e \in E(G^*)$. For every non-empty graph $F \subseteq K_3$, let $N(F,G^*,e)$ be the number of (induced unlabelled) copies of $F$ in $G^*$ containing $e$. Then
	\begin{equation}
	\label{eqn:1152am17mar21}
	\mathbb{E}_{G^*}(T)-\mathbb{E}_{G^*\setminus e}(T) \leq \Big(N(K_3,G^*,e) + N(K_{1,2},G^*,e)\,p_n + np_n^2\Big)(1-p_n).
	\end{equation}
By Definition \ref{def-core}(C3),
	\begin{equation}
	\label{eqn:1154am17mar21}
         t_n\leq N(K_3,G^*,e) + N(K_{1,2},G^*,e)\,p_n +np_n^2.
	\end{equation}
Since $p_n=\lambda/n$, and since there can be at most $n$ many $K_{1,2}$'s that contains a given edge $e$, we get $N(K_3,G^*,e) \geq t_n -C(\lambda)$. Moreover, if $e=uv$, then $N(K_3,G^*,e) \leq \min\{d_u,d_v\}$, where $d_u$ and $d_v$ are the degree of $u$ and $v$ in $G^*$ respectively. Therefore $\min\{d_u,d_v\} \geq  t_n -C(\lambda)$. Thus, both $u$ and $v$ are in $A_{\sss G^*}$.
	
For a set $A \subseteq [n]$, the number of $m$-cores $G^{*}$ that satisfy $A_{\sss G^{*}}\subseteq A$ is bounded above by ${{\ellalt^2}/{2} \choose m}$, where $|A|=\ellalt$. Therefore the number of cores with $m$ edges is bounded above by
	\begin{equation}
	\label{eqn:246pm21mar21}
	{n \choose \ellalt} {{\ellalt^2}/{2} \choose m} \leq n^{\ellalt}\left(\frac{\mathrm{e}\ellalt^2}{2m}\right)^{m}.
	\end{equation}
Since $\lim_{n\rightarrow \infty}\frac{\log{n}}{w_n^2k_n^{1/3}}=0$, we have $\lim_{n\rightarrow \infty}t_n =\infty$. Thus, $\ell\leq C' \frac{m}{t_n}$ for large enough $n$, and therefore
	\begin{equation*}
	n^\ellalt =\exp\left(\ellalt \log {n}\right) \leq \exp\left( C'  \tfrac{m}{t_n} \log {n}\right),
	\end{equation*}
and 
	\begin{equation*}
	\left(\tfrac{\mathrm{e}\ellalt^2}{2m}\right)^{m} = \exp\left(m\log\left(\tfrac{\mathrm{e}\ellalt^2}{2m}\right)\right) 
	\leq \exp\left(m\log\left(C'\tfrac{m}{t_n^2}\right)\right).
	\end{equation*}
Therefore the number of $m$-cores is at most $(1/p_n)^{m \Psi_n}$, where $\Psi_n$ is defined in \eqref{eqn:436pm21mar21}.
\end{proof}


\subsection{Upper bound on the upper tail probability}
\label{S4.2}

In this section, we use the lemmas proved in Section \ref{S4.1} to obtain a sharp bound on the upper tail probability, and thereby prove Theorem \ref{thm:607pm22mar21}.

\begin{proof}[Proof of Theorem \ref{thm:607pm22mar21}]
Fix a constant $C'>0$, and define $\psi_n$ as in \eqref{psindefextra}. By Lemma \ref{lem:117pm15mar21}, Lemma \ref{lem:317pm22mar21} and a union bound,
	\begin{equation}
	\begin{aligned}
	\mathbb{P}(T \geq \constant k_n) 
	&\leq (1+\xi_n)\mathbb{P}(G_{n,p_n} \text{ contains a seed}) \\
	&\leq (1+\xi_n)\mathbb{P}(G_{n,p_n} \text{ contains a core}) \\
	&\leq(1+\xi_n)\sum_{m \in \N_0} \mathbb{P}(G_{n,p_n} \text{ contains an } m\text{-core}).
	\end{aligned}
	\end{equation}
Again via a union bound and Lemma \ref{lem:307pm22mar21},
	\begin{equation}
	\label{eqn:327pm22mar21}
	\begin{aligned}
	\mathbb{P}(T \geq \constant k_n) 
	&\leq	(1+\xi_n)\sum_{m \in \N_0} p_n^{m}\,|\,\{G^*\subseteq K_n\colon\, G^*\text{ is an } m \text{-core}\}\,|\\
	&\leq (1+\xi_n)\sum_{m \geq m_{\min}} p_n^{(1-\Psi_n)m}.
	\end{aligned}
	\end{equation} 
Here, $m_{\min}$ is the minimum number of edges in the core. We next prove a lower bound of $m_{\min}$. First, recall from \eqref{eqn:1054pm10mar21} that $\Phi_{n,p_n,k_n} (\constant-2w_n) \leq e_{\sss G}\log(1/p_n)$ for any graph $G$ with $\mathbb{E}_{\sss G}(T)\geq (\constant-2w_n)k_n$. Further, since by Definition \ref{def-core}(C1), $\mathbb{E}_{\sss G^*}(T)\geq (\constant-2w_n)k_n$, we get  $\Phi_{n,p_n,k_n} (\constant-2w_n) \leq m_{\min}\log(1/p_n)$, and therefore 
	\begin{equation} 
	m_{\min} \geq \frac{\Phi_{n,p_n,k_n} (\constant-2w_n)}{\log(1/p_n)}.
	\end{equation} 
Also, using the fact that $m\leq C\constant w_n^{-1}k_n^{2/3}\log(1/p_n)$, and substituting $t_n := w_n^2 \frac{k_n^{1/3}}{C\constant \log(1/p_n)}$, we get
	\eqan{
	\label{eqn:242pm25mar21}
	\Psi_n &= \frac{1}{\log(1/p_n)}\left( \frac{1}{t_n} \log n 
	+ \log\left(C'\tfrac{m}{t_n^2}\right)\right) \leq \psin.
	}
By our assumptions we have $\psin=o(1)$. Now using \eqref{eqn:327pm22mar21}, we obtain
\begin{equation}
\label{eqn:505pm22mar21}
\begin{aligned}
\mathbb{P}(T \geq \constant k_n) &\leq (1+\xi_n) \frac{p_n^{(1-\psin)m_{\min}}}{1-p_n^{(1-\psin)}} 
\leq(1+\varepsilon)\exp\left(-(1-\psin) m_{\min}\log(1/p_n)\right)\\
&\leq (1+\varepsilon)\exp\left(- (1-\psin)\Phi_{n,p_n,k_n}(\constant-2w_n)\right),
\end{aligned}
\end{equation}
as required.
\end{proof}


\section{The variational problem}
\label{S5}

In this section, we prove \eqref{Phi-asymptotics} in Theorem \ref{thm-number-triangles}. Note that an explicit solution of the variational problem in \eqref{eqn:1054pm10mar21} would give rise to an explicit asymptotic formula for the upper tail probabilities in Theorems \ref{thm:437pm28apr21} and \ref{thm:607pm22mar21}. Indeed, we have already expressed the upper and lower bounds on the logarithmic probability of the event $\{T\geq \constant k_n\}$ in terms of $\Phi_{n,p,k}(\constant)$ with appropriate choices of the parameters. In this section, we obtain sharp estimates of the quantity $\Phi_{n,p,k}(\constant)$ for choices of parameters that cover both Theorems \ref{thm:437pm28apr21} and \ref{thm:607pm22mar21}. 

The following theorem provides asymptotically sharp upper and lower bounds on $\Phi_{n,p_n,k_n}(\constant)$ under a milder condition than those needed in Theorems \ref{thm:437pm28apr21} and \ref{thm:607pm22mar21}.

\begin{thm}{\bf [Asymptotics for the variational problem]}
\label{thm:944pm28apr21}
$\mbox{}$\\ 
For every sequence of positive reals $(w_n)_{n\in\N}$ such that $\lim_{n\to\infty} w_n = 0$ and $\lim_{n\to\infty}$ $w_nk_n^{1/3}=\infty$, for $n$ large enough,
	\begin{equation}
	\tfrac12[6\constant(1-w_n)k_n]^{2/3} \log(1/p_n) \leq \Phi_{n,p_n,k_n}(\constant) 
	\leq \tfrac12[6\constant(1+w_n)k_n]^{2/3} \log (1/p_n).
	\end{equation}
\end{thm}

\proof
Let $(w_n)_{n\in\N}$ be such that $\lim_{n\to\infty} w_n = 0$ and $\lim_{n\to\infty} w_nk_n^{1/3} = \infty$. Note that if $G$ is a complete graph on $(6\constant(1+w_n)k_n)^{1/3}$ vertices, then $\mathbb{E}_{\sss G}(T)\geq \constant k_n$ for $n$ sufficiently large. Moreover, we can find such a complete graph even when $6\constant k_n$ is not a cube of some number, since the nearest cube larger than $6\constant k_n$ is within $6\constant k_n+C'k_n^{2/3}$ for some constant $C'$. Therefore,
	\begin{equation}
	\label{eqn:411pm23mar21}
	\Phi_{n,p_n,k_n}(\constant) \leq \tfrac12 (6\constant(1+w_n)k_n)^{2/3}\log(1/p_n).
	\end{equation}
For the lower bound, suppose that $G$ is a graph that satisfies $\mathbb{E}_{\sss G}(T)\geq \constant k_n$. Then $\Phi_{n,p,k}(\constant) \geq e_{\sss G} \log\left(1/p\right)$. We are now left to find an appropriate lower bound on $e_{\sss G}$. Since $\mathbb{E}_{\sss G}(T)\geq \constant k_n$,
	\begin{equation}
	\constant k_n \leq \mathbb{E}_{\sss G}(T) \leq N(K_{1,2},G)\,p_n+N(K_2,G)\,np_n^2+N(K_3,G)+ \tfrac16 \lambda^3.
	\end{equation}
Here, $N(K_{1,2},G) \leq N(K_2,G) |V(G)|$ (since each edge in $G$ can be in at most $|V(G)|$ many $K_{1,2}$'s), and therefore
	\begin{equation}
	\label{eqn:1112pm24mar21}
	\begin{aligned}
	\constant k_n &\leq \mathbb{E}_{\sss G}(T) \leq  N(K_2,G) |V(G)|\,p_n+N(K_2,G)\,np_n^2+N(K_3,G) +\tfrac16\lambda^3\\
	&\leq \left(\lambda+ n^{-1}\lambda^2\right)N(K_2,G) + N(K_3,G)+\tfrac16 \lambda^3.
	\end{aligned}
	\end{equation}
Now assume by contradiction that $e_{\sss G}=N(K_2,G) \leq \tfrac12(6\constant(1-w_n)k_n)^{2/3}$. Then Lemma \ref{lem:1109pm24mar21} gives $N(K_3,G) \leq \constant (1-w_n)k_n$, and the right-hand side of \eqref{eqn:1112pm24mar21} is bounded above by $\constant (1-w_n)k_n+ \tfrac12(6\constant(1-w_n)k_n)^{2/3}= \constant k_n -r_n$, where $\lim_{n\to\infty} r_n = \infty$ by our assumption that $\lim_{n\to\infty} w_nk_n^{1/3} = \infty$, which contradicts \eqref{eqn:1112pm24mar21}. Therefore $e_{\sss G}=N(K_2,G) \geq \tfrac12(6\constant(1-w_n)k_n)^{2/3}$, which together with $\Phi_{n,p,k}(\constant) \geq e_{\sss G} \log\left(1/p\right)$ completes the proof of the lower bound.
\qed


\section{Structure of random graphs with many triangles}
\label{S6}

In this section, we investigate the structure of $G_{n,p_n}$ {\em conditionally} on $T\geq \constant k_n$ as stated in Theorem \ref{thm-structure-graph-many-triangles}. In Section \ref{sec-near-clique-many-triangles}, we show that, conditionally on $T\geq \constant k_n$, there is a near-clique (recall Definition \ref{def-near-clique}). In Section \ref{sec-local-structure-many-triangles}, we study the local structure of the Erd\H{o}s-R\'enyi random graph conditionally on $T\geq \constant k_n$, and show that its local limit exists and is the same as that for the original Erd\H{o}s-R\'enyi random graph.


\subsection{Localization of the triangles: Existence of near-cliques}
\label{sec-near-clique-many-triangles}

The following theorem states a more precise version of \eqref{existence-near-clique} in Theorem \ref{thm-structure-graph-many-triangles}, where we recall the notion of a $(\delta,m)$-clique in Definition \ref{def-near-clique}:

\begin{thm}{\bf [Existence of a near-clique]}
\label{thm:1000am24jun21}
Under the conditions of Theorem~\ref{thm:607pm22mar21}, for every $\delta>0$, there exists a constant $C>0$ such that, for $n$ large enough,
	\begin{equation}
	\mathbb{P}\left(G_{n,p_n} \text{\rm ~contains no } \left(C(\psin+w_n),m_n\right)\text{\rm-clique} \mid T\geq \constant k_n\right) 
	\leq\delta, 
	\end{equation}
where $m_n:= \tfrac12[6(\constant-2w_n)k_n(1-Ck_n^{-1/3})]^{2/3}$ and $\psi_n$ is defined in \eqref{psindefextra}.
\end{thm}

\noindent
The proof will follow from Propositions \ref{thm:259pm25mar21} and \ref{thm:804pm23jun21} below. 

We start by setting the stage. We will work with $\psin$-minimal cores, which are cores that do not have too many edges (recall also Definition \ref{def-core} for the definition of a core).
	
\begin{dfn}{\bf [$\psin$-minimal cores]}
For a sequence $(\psi_n)_{n\in\N}$ with $\lim_{n\to\infty} \psi_n = 0$, $G$ is called a $\psi_n$-minimal core if $G$ is a core with $e_{\sss G} \leq \tfrac12 (6\constant(1+\psi_n)k_n)^{2/3}$. 
\end{dfn}

\noindent
Note that the upper bound in $e_{\sss G} \leq \tfrac12 (6\constant(1+\psi_n)k_n)^{2/3}$ is chosen in relation to the event $T\geq \constant k_n$ of interest. The following lemma reveals the structure of the graph when it has many triangles:

\begin{prop}{\bf [Unlikely not to have a $\psi_n$-minimal core]}
\label{thm:259pm25mar21}
$\mbox{}$\\
For every $\constant,\delta>0$ and every sequence $(w_n)_{n\in\N}$ in Definition \ref{def-core} of the core satisfying $\lim_{n\to\infty}\frac{\log{n}}{w_n^2k_n^{1/3}}$ $=0$ and $\lim_{n\to\infty} (\log w_n)/\log(1/p_n)=0$,
	\begin{equation}
	\mathbb{P}\left(T\geq \constant k_n,\, G_{n,p_n} \text{\rm ~contains no } 2\psin\text{\em -minimal core}\right) 
	\leq \delta\,\mathbb{P}(T\geq \constant k_n).
	\end{equation}
\end{prop}

\begin{proof}
For any sequence of positive reals $\psi_n'$, we can split
	\begin{equation}
	\begin{aligned}
	&\mathbb{P}\left(T\geq \constant k_n,\,G_{n,p_n} \text{ contains no } \psi_n'\text{-minimal core}\right)\\ 
	&\leq \mathbb{P}\left(T\geq \constant k_n, G_{n,p_n} \text{ contains no core}\right) 
	+ \mathbb{P}\left(G_{n,p_n} \text{ has a core that is not } \psi_n'\text{-minimal}\right)\\
	&= \text{Term-I} + \text{Term-II}.
	\end{aligned}
	\end{equation}
We bound both contributions. For \textbf{Term-I}, we start by recalling that \eqref{eqn:242pm04oct21} in Lemma \ref{lem:117pm15mar21} implies
	\begin{equation}
	\mathbb{P}\left(T\geq \constant k_n,\, G_{n,p_n} \text{ contains no seed}\right) 
	\leq\xi_n\,\mathbb{P}(T\geq \constant k_n),
	\end{equation}
where 
	\begin{equation}
	\xi_n:= \exp\left([\tfrac12 (6(\constant -w_n))^{2/3}-\tfrac13C]\,k_n^{2/3}\log(1/p_n)\right).
	\end{equation} 
By Lemma \ref{lem:317pm22mar21},
	\begin{equation}
	\label{eqn:210pm25mar21}
	\mathbb{P}\left(T\geq \constant k_n,\,G_{n,p_n} \text{\rm contains no core}\right) 
	\leq \xi_n\,\mathbb{P}(T\geq \constant k_n).
	\end{equation}
Hence we can use \eqref{eqn:210pm25mar21} to get
	\begin{equation}
	\label{eqn:855pm05mar21}
	\text{Term-I} \leq \tfrac12 \delta\,\mathbb{P}(T\geq \constant k_n).
	\end{equation}
This bounds \textbf{Term-I}. For \textbf{Term-II}, using the union bound, we have
	\begin{equation}
	\text{Term-II} \leq \sum_{m \geq \frac{1}{2}(6\constant(1+\psi_n')k_n)^{2/3}} p_n^{m} 
	\,|\,\{G^*\subseteq K_n\colon\, G^*\text{ is an }m\text{-core}\}\,|.
	\end{equation}
We take $\psi_n'=2\psin$, recall \eqref{eqn:242pm25mar21}, and note that, under the assumptions in Proposition \ref{thm:259pm25mar21}, $\lim_{n\to\infty} \psin $ $=0$. By Lemma~\ref{lem:307pm22mar21} (see, in particular, \eqref{eqn:327pm22mar21} in its proof),
	\begin{equation}
	\label{eqn:254pm25mar21}
	\begin{aligned}	
	\text{Term-II} \leq \sum_{m \geq \tfrac12(6\constant(1+2\psin)k_n)^{2/3}} p_n^{m(1-\psin)} 
	\leq (1+\varepsilon)\,p_n^{\tfrac12(6\constant(1+2\psin)k_n)^{2/3}(1-\psin)}. 
	\end{aligned}
	\end{equation}
Note that $(1+2x)^{2/3}(1-x)=((1+2x)(1-x))^{2/3}(1-x)^{1/3}$, which gives $(1+2x)^{2/3}(1-x) \geq (1+0.9x)^{2/3}(1-x)^{1/3} \geq (1+0.7x)^{1/3} \geq 1+0.2x$ for $x$ sufficiently small. Since $\lim_{n\to\infty} \psin = 0$, using this in \eqref{eqn:254pm25mar21}, we get
	\begin{equation}
	\label{eqn:931pm25mar21}
	\text{Term-II} \leq  p_n^{\tfrac12(6\constant k_n)^{2/3}+c\psin} 
	\leq \tfrac12 \varepsilon\, p_n^{\tfrac12 (6\constant k_n)^{2/3}} \leq \tfrac12\delta\,\mathbb{P}(T\geq \constant k_n).
	\end{equation}
In the last equation we use that $\psin\geq \frac{\log{C'(\log(1/p_n))}}{\log(1/p_n)}$, which gives $ p_n^{\psin} \leq \exp{(-\log{C'(\log(1/p_n))})}$, and so $\lim_{n\rightarrow \infty}  p_n^{\psin} =0$. This bounds \textbf{Term-II}.
	
Combining \eqref{eqn:855pm05mar21} and \eqref{eqn:931pm25mar21},  we arrive at the claim.
\end{proof}

We continue by studying subgraphs of $\psin$-minimal cores. Proposition \ref{thm:804pm23jun21} below proves that a subgraph exists with a high minimal degree. Let $\mathrm{Emb}(H,G)$ denote the set of embeddings of a graph $H$ in $G$, i.e., the set of edge-preserving injective maps from $V(H)$ to $V(G)$. We use the following theorem from \cite{keevash2008shadows} (also proved in \cite{harel2019upper}):

\begin{prop}{\bf [Existence of subgraph with high minimal-degree under embeddings bound]}
\label{keevash_stab}
If a graph $G$ satisfies
	\begin{equation}
	|\mathrm{Emb}(K_3,G)|\geq (1-\delta)(2e_{\sss G})^{3/2}
	\end{equation}
for some $\delta\geq (e_{\sss G})^{-1/2}$, then $G$ has a subgraph $G'$ with minimum degree $(1-4\delta^{1/2})(2e_{\sss G})^{1/2}$.
\end{prop}

\begin{prop}{\bf [Existence of subgraph of $\psin$-minimal core with high minimal-degree]}
\label{thm:804pm23jun21}
A $\psin$-minimal core $G$ contains a subgraph $G'$ whose minimum degree is at least $(1-4(C(\psin+w_n))^{1/2})(2e_{\sss G})^{1/2}$.
\end{prop}

\begin{proof}
A $\psin$-minimal core has at most $\tfrac12(6\constant(1+\psin)k_n)^{2/3}$ edges and, by Definition \ref{def-core} of a core, also $\mathbb{E}_{\sss G}(T) \geq (\constant-2w_n)k_n$. Therefore
	\begin{equation}
	\begin{aligned}
	(\constant-2w_n)k_n \leq \mathbb{E}_{\sss G}(T) 
	&\leq N(K_{1,2},G)\,p_n+N(K_2,G)\,np_n^2+N(K_3,G)+ \tfrac16\lambda^3\\
	&\leq \tfrac12 \lambda(6\constant(1+\psin)k_n)^{2/3}
	+\tfrac12n^{-1}\lambda^2(6\constant(1+\psin)k_n)^{2/3}+N(K_3,G)+ \tfrac16 \lambda^3.
	\end{aligned}
	\end{equation}
In the last display we have used that $N(K_{1,2},G) \leq N(K_2,G) |V(G)|$, since each edge in $G$ can be in at most $|V(G)|$ many $K_{1,2}$'s. Now bounding $\psin$ by $1$, we get
	\begin{equation}
	\label{eqn:705pm23jun21}
	N(K_3,G) \geq (\constant-2w_n)k_n- \left(\tfrac12 \lambda+\tfrac12 n^{-1}\lambda^2\right)(12\constant k_n)^{2/3} - \tfrac16 \lambda^3.
	\end{equation}
Write $e_{\sss G}^{\max}:= \tfrac12(6\constant(1+\psin)k_n)^{2/3}$ (which is the maximum number of edges a $\psin$-minimal core can have). Also note that $|\mathrm{Emb}(K_3,G)| = 6N(K_3,G)$. Therefore \eqref{eqn:705pm23jun21} gives
	\begin{equation}
	\label{eqn:757pm23jun21}
	\begin{aligned}
	|\mathrm{Emb}(K_3,G)|&\geq (2e_{\sss G}^{\max})^{3/2}\left(1-C(\psin+w_n)-Ck_n^{-1/3}\right)\\
	&\geq (2e_{\sss G}^{\max})^{3/2}\left(1-C(\psin+w_n)\right)\\
	&\geq (2e_{\sss G})^{3/2}\left(1-C(\psin+w_n)\right),
	\end{aligned}
	\end{equation}
where in the second line we enlarge the constant $C$ and use the fact that $\lim_{n\to\infty} k_n^{1/3}\psin=\infty$. Now, \eqref{eqn:705pm23jun21} and Lemma \ref{lem:1109pm24mar21} give that the core $G$ must have at least $e_{\sss G}^{\min}:= \tfrac12[6(\constant-2w_n)k_n(1-Ck_n^{-1/3})]^{2/3}$ many edges. Clearly, $C(\psin+w_n)\geq (e_{\sss G}^{\min})^{-1/2} \geq (e_{\sss G})^{-1/2}$. Therefore, using \eqref{eqn:757pm23jun21}, we get from Proposition \ref{keevash_stab} that $G$ contains a subgaph $G'$ with minimum degree $(1-4(C(\psin+w_n))^{1/2})(2e_{\sss G})^{1/2}$.
\end{proof}


\subsection{The local limit of the random graph conditionally on many triangles}
\label{sec-local-structure-many-triangles}

In this section, we show that even if we condition the graph to have at least $\constant k_n$ triangles with $k_n$ not too large, then the local structure of the graph remains unaffected. To describe the result, we introduce some notation and define local convergence. We refer to \cite{aldous2007processes} or \cite[Chapter 2]{Hofs21} for more details. 

A \emph{rooted graph} $(G,o)$ is a locally-finite (i.e., the degrees of all the vertices are finite) and connected graph $G=(V,E)$ with a distinguished vertex $o\in V$, called the root. For an integer $r\geq 0$, let $(G,o)_r$ be the induced subgraph on the vertex set $\{u\in V\colon \dist_{\sss G}(o,u)\leq r\}$, where $\dist_{\sss G}$ is the usual graph distance in $G$. 

Two rooted graphs $(G_i,o_i)=(V_i,E_i,o_i)$, $i=1,2$, are {\em isomorphic} if there exists a bijection $\sigma\colon\,V_1\rightarrow V_2$ such that $\sigma(o_1)=o_2$, and $\{u,v\}\in E_1$ if and only if $\{\sigma(u),\sigma(v)\}\in E_2$. It is easy to check that the rooted isomorphism is an equivalence relation, which we denote by $(G_1,o_1)\simeq (G_2,o_2)$. An equivalence class of rooted graphs is sometimes just called an unlabelled rooted graph. 

Let $\mathscr{G}^\star$ be the set of all locally finite, connected and unlabelled rooted graphs. For $\gamma \in \Gstar$ and an integer $h \geq 0$, let $\gamma_h$ denote the unlabelled rooted graph obtained by removing all the vertices, and all the edges incident to them, that are at graph-distance larger than $h$ from the root. 

The {\em local topology} is the smallest topology such that for any $\gamma \in \Gstar$ the $\Gstar \to \{0,1\}$ function $f(g)= \mathbf{1}(g_h=\gamma_h)$ is continuous. This topology is metrizable, and with this topology the space $\Gstar$ is complete and separable \cite{aldous2007processes}. Write $\mathcal{P}(\Gstar)$ to denote the space of probability measures on $\Gstar$, and equip $\mathcal{P}(\Gstar)$ with the topology of weak convergence.

 Fix a finite graph $G=(V,E)$. The empirical neighborhood distribution of $U(G) \in \mathcal{P}(\Gstar)$ is defined by
	\begin{equation}
	\label{eqn:915am28sep}
	U(G):=\frac{1}{|V|} \sum_{v\in V} \delta_{[G,v]},
	\end{equation}
where $[G,v] \in \Gstar$ denotes the equivalence class of $(G(v),v)$ and $\delta_g$ is the Dirac mass at $g \in \Gstar$. For a sequence of finite graphs $\{G_n\}_{n \in \mathbb{N}}$, we say that $G_n$ converges to $\rho \in \mathcal{P}(\Gstar)$ in the \emph{local weak sense} if $U(G_n)$ converges to $\rho$ in $\mathcal{P}(\Gstar)$, i.e., if for all bounded and continuous functions $h\colon \Gstar \to \mathbb{R}$,
	\eqn{
	\mathbb{E}\Big[\frac{1}{|V|} \sum_{v\in V} h([G,v])\Big]\rightarrow \mathbb{E}_\rho[h(G,o)],
	}
where the expectation is w.r.t.\ the (possibly random) graph $G_n$. We further say that $G_n$ converges {\em locally in probability} to $\rho \in \mathcal{P}(\Gstar)$ if for all bounded and continuous functions $h\colon \Gstar \to \mathbb{R}$,
	\eqn{
	\frac{1}{|V|} \sum_{v\in V} h([G,v])\stackrel{\mathbb{P}}{\rightarrow} \mathbb{E}_\rho[h(G,o)].
	}
Now consider $\text{UGW}(\text{Poi}(\lambda)) \in \mathcal{P}(\Gstar)$, the Galton-Watson tree where the root has $\text{Poi}(\lambda)$ many children, and all the children have $\text{Poi}(\lambda)$ many children independently, and so on. Our main result on local convergence is the following theorem:

\begin{thm}{\bf [Local convergence conditionally on many triangles]}
\label{thm:428pm11oct21}
 Let $(k_n)_{n\in \N}$ be such that $\lim_{n \to \infty} k_n^{2/3} \log (1/p_n)/n=0$. Then, under the conditions of Theorem \ref{thm:437pm28apr21}, conditionally on $T\geq \constant k_n$, $G_{n,p_n}$ converges locally in probability to $\text{\rm UGW}(\text{\rm Poi}(\lambda))$.
\end{thm}

Our proof crucially relies on \cite[Theorem 1.8]{bordenave2015large}. In order to show that the local structure of the random graph does not change under the presence of a large number of triangles, we compute the probability that the random graph is concentrated around the Galton-Watson tree with offspring distribution Poi$(\lambda)$. We rely on the following result from \cite{bordenave2015large}:

\begin{prop}{\bf \cite[Theorem 1.8]{bordenave2015large}}
\label{thm:341pm30jun21}
Fix $\lambda>0$, and let $G_{n,p_n}$ be the random graph on $n$ vertices with $p_n=\lambda/n$. Then $U(G_{n,p_n})$ satisfies the large deviation principle in $\mathcal{P}(\Gstar)$ with rate $n$ and with a good rate function $I$, i.e., $I$ is lower semi-continuous and has compact level sets. Further,
	\begin{equation}
	\label{eqn:330pmjun21}
		I(\rho) = \tfrac12 \lambda -\tfrac12 d\log{\lambda} - \Sigma\left(\rho\right), \qquad \rho\in\mathcal{P}(\Gstar),
	\end{equation}
where $d:= \mathbbm{E}_{\rho}(\text{\rm deg}_{\sss G}(o))$ is the degree of $o$ in $G$, with the convention that $\Sigma\left(\rho\right) =0$ if $d=0$. The quantity $\Sigma\left(\rho\right)$ appearing in the rate function will be defined below.
\end{prop}

To describe $\Sigma\left(\rho\right)$ appearing in the rate function, we need some further notation. Let $\mathcal{G}_{n,m}$ be the set of graphs on vertex set $[n]$ and $m$ edges. If $\rho \in \mathcal{P}(\Gstar)$, then define
	\begin{equation}
	\mathcal{G}_{n,m}(\rho, \varepsilon):=\{G \in \mathcal{G}_{n,m}: U(G) \in B(\rho, \varepsilon)\},
	\end{equation}
where $B(\rho, \varepsilon)$ denotes the $\varepsilon$-open ball around $\rho$ in L\'{e}vy-Prokhorov metric on $\mathcal{P}(\Gstar)$. 

For the convenience of the reader we recall the definition of the L\'{e}vy-Prokhorov metric. Consider a subset $A \subseteq \mathcal{B}(\Gstar)$, where $\mathcal{B}(\Gstar)$ be the Borel sigma-field of the space $\Gstar$ equipped with local topology, and define its $\varepsilon$-neighborhood by
	\begin{equation}
	A^{\varepsilon}=\cup_{g \in A}{B}_{\Gstar}(g,\varepsilon),
	\end{equation}  
where ${B}_{\Gstar}(g,\varepsilon)$ is an open ball in $\Gstar$ of radius $\varepsilon$ around $g$. Then the L\'{e}vy-Prokhorov metric $\pi\colon \mathcal{P}(\Gstar)^2 \to [0,\infty)$ is defined as follows: For $\rho_1, \rho_2\in \mathcal{P}(\Gstar)$,
	\begin{equation}
	\pi(\rho_1,\rho_2):= \inf\{\varepsilon>0:\rho_1(\constant) \leq \rho_2(A^\varepsilon)+\varepsilon \text{  and  } 
	\rho_2(\constant) \leq \rho_1(A^\varepsilon) + \varepsilon \text{  for all  }A\in \mathcal{B}(\Gstar)\}.
	\end{equation} 
Since $\Gstar$ is separable in the local topology, the topology induced by $\pi$ on $\mathcal{P}(\Gstar)$ is equivalent to the topology of weak convergence on $\mathcal{P}(\Gstar)$ (see \cite{billingsley1999convergence}). We write
	\begin{equation}
	\Sigma(\rho,\varepsilon)=\lim_{n\rightarrow \infty} \frac{\log{|\mathcal{G}_{n,m}(\rho, \varepsilon)|}-m\log{n}}{n}.
	\end{equation}
Finally, 
	\begin{equation}
	\Sigma(\rho) := \lim_{\varepsilon\searrow 0} \Sigma(\rho,\varepsilon).
	\end{equation} 
Write $s(d)=\frac{d}{2}-\frac{d}{2} \log{d}$. We are now ready to state and prove the main proposition in this section, which will immediately prove Theorem \ref{thm:428pm11oct21}:

\begin{prop}{\bf [Exponential concentration around $\text{\rm UGW} ({\rm Poi}(\lambda))$]}
\label{lem:433pm12sep21}
	For every $\varepsilon>0$, and $n\geq n(\varepsilon)$, there exists a constant $C_{\varepsilon}>0$ such that
	\begin{equation}
	\label{eqn:432pm12sep21}
		\mathbb{P}\left(U(G_n)\in B^c(\text{\rm UGW} ({\rm Poi}(\lambda)), \varepsilon) \right) \leq \exp\left(-C_{\varepsilon}n\right).
	\end{equation}
\end{prop}

\begin{proof}

We use the exponential tightness of a sequence of probability measures satisfying a large deviation upper bound with a rate function that has compact level sets, as stated in Proposition \ref{thm:341pm30jun21}. This implies that, for each positive integer $l\geq 1$, there is a compact set $C_l \subseteq \mathcal{P}(\Gstar)$ such that 
	\begin{equation}
	\label{eqn:226pm12sep21}
		\limsup_{n\rightarrow \infty}\frac{1}{n}\log \mathbb{P}\left(\rho \in C_l^c \right) \leq -l.
	\end{equation}
Now let $S_l=C_l\cap B^c(\text{\rm UGW}({\rm Poi}(\lambda)), \varepsilon)$. Then, for every fixed $\varepsilon>0$ and $n\geq n(\varepsilon)$,
	\begin{equation}
	\mathbb{P}\left(U(G_n)\in B^c(\text{\rm UGW}({\rm Poi}(\lambda)), \varepsilon) \right) \leq \mathbb{P}\left(U(G_n)\in S_l \right) +\exp(-C_l n),
	\end{equation} 
for some constant $C_l>0$. Note that $S_l$ is a compact set (since it is a closed subset of a compact set) Using Proposition \ref{thm:341pm30jun21} again we have, for any $\delta>0$ and $n \geq n(\delta)$,
	\begin{align}
	\frac{1}{n}	\log \mathbb{P}\left(U(G_n)\in S_l \right) &\leq -\inf_{\rho \in S_l} I(\rho)+\delta \notag \\
	&= -\min_{\rho \in S_l} I(\rho)+\delta.
	\end{align}
In the second line, we have used the fact that $I$ is lower semi-continuous, and therefore the minimum is attained on a compact set. 

We are done if we can show that $\min_{\rho \in S_l}I(\rho)>0$. First, let us recall from \cite[Remark 4.12]{backhausz2021typicality} that UGW$(\text{\rm Poi}(d))$ {\em uniquely} maximises $\Sigma(\rho)$ among all measures $\rho$ with mean degree $d$. (Note that this does not directly follow from \cite{bordenave2015large}. In particular, in \cite[Theorem 1.3]{bordenave2015large} below equation (9), the authors assume that $\rho_1$ has finite support, which can be removed due to \cite[Remark 4.12]{backhausz2021typicality}.) Also recall from \cite[Theorem 1.2]{bordenave2015large} that $\Sigma(\rho) \leq s(d)$, where $s(d)=\frac{d}{2}-\frac{d}{2} \log{d}$. Therefore, $\Sigma(\rho)<s(d)$ for all $\rho$ except $\rho$= UGW$(\text{Poi}(d))$.

Let $\rho'$ be a minimiser of $I$ on $S_l$ having mean degree $d$. Since $\rho' \neq$ UGW$(\text{Poi}(\lambda))$, there can be two cases:

\medskip\noindent
	\textbf{Case I:} $\rho'$ has mean degree $d\neq \lambda$.\par
	In this case,
	\begin{equation}
	\begin{aligned}
		\min_{\rho \in S_l}I(\rho)&=  \tfrac12 \lambda -\tfrac12 d\log{\lambda} - \max_{\rho \in S_l}\Sigma\left(\rho\right) \\
		&\geq \tfrac12 \lambda -\tfrac12d \log{\lambda} - s(d) \\
		&= \tfrac12 \lambda -\tfrac12 d \log{\lambda} -\tfrac12 d+\tfrac12 d\log{d}>0.
	\end{aligned}
	\end{equation}
	
\medskip\noindent
	\textbf{Case II:} $\rho'$ has mean degree $d=\lambda$.\par
	In this case, since $\rho' \neq$ UGW$(\text{\rm Poi}(\lambda))$, we must have $\Sigma(\rho')<s(\lambda)$, and therefore 
	\begin{equation}
	\begin{aligned}
		\min_{\rho \in S_l}I(\rho)&=  \tfrac12 \lambda -\tfrac12 \lambda \log{\lambda} - \max_{\rho \in S_l}\Sigma\left(\rho\right) \\
		&=  \tfrac12 \lambda -\tfrac12 \lambda \log{\lambda} - \Sigma\left(\rho'\right) \\
		&> \tfrac12 \lambda -\tfrac12 \lambda \log{\lambda} -s(\lambda)=0,
	\end{aligned}
	\end{equation}
which completes the proof.
\end{proof}

We can now complete the proof of Theorem \ref{thm:428pm11oct21}:

\begin{proof}[Proof of Theorem \ref{thm:428pm11oct21}] 
First note that
	\begin{equation}
	\label{eqn:435pm11oct21}
		\mathbb{P}\left(U(G_n)\in B^c(\text{UGW}({\rm Poi}(\lambda)), \varepsilon) \mid T\geq \constant k_n\right)
		\leq \frac{\mathbb{P}\left(U(G_n)\in B^c(\text{\rm UGW}({\rm Poi}(\lambda)), \varepsilon)\right)}
		{\mathbb{P}\left(T\geq \constant k_n\right)}.
	\end{equation}
We can upper bound the numerator in the right-hand side of \eqref{eqn:435pm11oct21} by Proposition \ref{lem:433pm12sep21}, and lower bound the denominator by Theorems \ref{thm:437pm28apr21} and \ref{thm:944pm28apr21}, to obtain
	\begin{equation}
	\begin{aligned}
	\mathbb{P}\left(U(G_n)\in B^c(\text{UGW}({\rm Poi}(\lambda)), \varepsilon) \mid T\geq \constant k_n\right)
	&\leq\frac{\exp(-C_{\varepsilon}n)}{\exp(-(1+\varepsilon_n)\tfrac12[6\constant(1+w_n)k_n]^{2/3} \log (1/p_n))}\\
	&= \exp\big(-C_{\varepsilon}n + C_a(1+o(1)) k_n^{2/3} \log (1/p_n)\big),
	\end{aligned}
	\end{equation}
where we use that $\varepsilon_n = o(1)$ and $w_n = o(1)$. The right-hand side tends to zero exponentially fast because of our assumption that $\lim_{n \to \infty} k_n^{2/3} \log (1/p_n)/n=0$.
\end{proof}


\section{Large deviations for the number of vertices in triangles}
\label{S7}

In this section, we prove Theorem \ref{thm-many-vertices-in-triangles} by providing sharp upper and lower bounds of the probability that at least $k_n$ vertices are part of some triangle in $G_{n,p_n}$. Roughly, the proof is based on the idea that if $G_{n,p_n}$ contains $k_n$ vertices that are part of some triangle, then there are approximately $k_n/3$ disjoint triangles in the graph. 

The proof of the lower bound in Section~\ref{S7.1} is relatively easy and involves computing an appropriate lower bound of the probability of the aforementioned event. The upper bound in Section~\ref{S7.2} is much more delicate. For this we carefully compute the number of ways a graph on $n$ vertices can contain $k_n$ many vertices that are part of a triangle. To do this, we find it more convenient to work with a `basic graph'. Roughly, basic graphs do not contain any {\em unnecessary edges}. More precisely, the edges that are not contributing to the formation of any triangle in the graph are removed. Interestingly, these basic graphs admit a structural decomposition, which in turn enables us to compute the number of basic graphs with a given number of edges. Using this fact, we compute the upper bound of the probability that in $G_{n,p_n}$ the $k_n$ vertices are part of some triangle. The details of the proof are presented in Section~\ref{S7.2}.


\subsection{Proof of the lower bound in Theorem \ref{thm-many-vertices-in-triangles}}
\label{S7.1}

In this section, we prove the lower bound in Theorem \ref{thm-many-vertices-in-triangles}, by proving the following more precise theorem:

\begin{thm}{\bf [Large deviation lower bound on the number of vertices in triangles]}
\label{thm:521pm10sep21}
$\mbox{}$\\
Let $(k_n)_{n\in\N}$ be such that $\lim_{n\to\infty} k_n = \infty$ and $k_n \leq n$. Then, there exists a constant $C>0$ such that, for $n$ large enough,
	\begin{equation}
	\log	\mathbb{P}(V_{\sss T}(G_{n,p_n})\geq k_n) \geq -\tfrac13 k_n\log(\tfrac13 k_n)-Ck_n.
	\end{equation}
\end{thm}

For the lower bound on $\mathbb{P}(V_{\sss T}(G_{n,p_n})\geq k_n)$, we construct an appropriate lower bounding event. For the rest of the proof, we assume that $k_n$ is divisible by $3$. It is easy to see that this does not affect our estimates. Indeed, if $k_n$ is not divisible by $3$, then $k_n+t$ must be divisible by $3$, where $t=1$ or $2$, and then $\mathbb{P}(V_{\sss T}(G_{n,p_n})\geq k_n) \geq \mathbb{P}(V_{\sss T}(G_{n,p_n})\geq k_n+t)$, and the contribution of $t$ is easily seen to be negligible. 

Define the event $\mathcal{T}_S$ by
	\begin{equation}
	\label{eqn:623pm24may21}
	\begin{aligned}	
	\mathcal{T}_S
	&:= \{S \text{ is a union of vertex-disjoint triangles and no other edge is present in }G[S]\} \\ 
	&\qquad\qquad \cap\{G\setminus E(G[S])\text{ is triangle-free}\},
	\end{aligned}
	\end{equation}
where we recall that $G[S]$ is the subgraph of $G$ induced by the vertex set $S$. Note that the definition of $\mathcal{T}_S$ is valid only when $|S|$ is divisible by $3$, and that $\mathcal{T}_S$ and $\mathcal{T}_{S'}$ are {\em disjoint} for any different sets $S, S'\subseteq [n]$.

\begin{lem}{\bf [Only vertex-disjoint triangles]}
\label{lem:210pm23jun21}
For $S \subseteq [n]$ with $|S|=k_n$,
	\begin{equation}
	\label{eqn:210pm23jun21}
	\mathbb{P}(\mathcal{T}_S) \geq \frac{k_n!}{(3!)^{\tfrac13 k_n}(\tfrac13 k_n)!} 
	p_n^{k_n}(1-p_n)^{{k_n \choose 2}- k_n + (n-k_n)k_n}\, 
	\tfrac{9}{10}\exp\left(-\tfrac16\lambda^3\right).
	\end{equation}
\end{lem}

\begin{proof}
First observe that
	\begin{equation}
	A_S\cap B_S \cap C_S \subseteq \mathcal{T}_S,
	\end{equation}
where 
	\begin{equation}
	\begin{aligned}
	A_S &:= \{[S] \text{ is a union of vertex-disjoint triangles and no other edge is present in } G[S]\},\\
	B_S &:= \{\text{there is no edge between } V\setminus S \text{ and } S\},\\ 
	C_S &:= \{G[V \setminus S] \text{ is triangle-free}\}.
	\end{aligned}
	\end{equation} 
Then $A_S$, $B_S$, $C_S$ are independent events, so that
	\begin{equation}
	\mathbb{P}(\mathcal{T}_S) \geq \mathbb{P}(A_S) \mathbb{P}(B_S) \mathbb{P}(C_S).
	\end{equation}
Since the distribution of the number of triangles converges to a Poisson distribution with mean $\frac16\lambda^3$, for large enough $n$ we have $\mathbb{P}(C_S) \geq \tfrac{9}{10}\exp\left(-\frac16\lambda^3\right)$. It is easy to see that $\mathbb{P}(B_S)=(1-p_n)^{(n-k_n)k_n}$. Finally, for $\mathbb{P}(A_S)$ we assume that $k_n$ is divisible by $3$, and compute
	\begin{equation}
	\mathbb{P}(A_S) = \frac{k_n!}{(3!)^{\tfrac13 k_n} (\tfrac13 k_n)!}\,p_n^{k_n}(1-p_n)^{{k_n\choose 2}-k_n}.
	\end{equation}
We combine this with
	\begin{equation}
	\mathbb{P}(\mathcal{T}_S) \geq \frac{k_n!}{(3!)^{\tfrac13 k_n} (\tfrac13 k_n)!}\,
	p_n^{k_n}(1-p_n)^{{k_n\choose 2}-k_n+ (n-k_n)k_n}\,\tfrac{9}{10}\,\exp\left(-\tfrac16\lambda^3\right)
	\end{equation}
to arrive at the claim.
\end{proof}

Now we are ready to complete the proof of Theorem \ref{thm:521pm10sep21}:

\begin{proof}[Proof of Theorem \ref{thm:521pm10sep21}]
Note that
	\begin{equation}
	\mathbb{P}(V_{\sss T}(G_{n,p_n})\geq k_n) \geq \mathbb{P}(\cup_{|S|=k_n}\mathcal{T}_S) 
	= \sum_{|S|=k_n}\mathbb{P}(\mathcal{T}_S).
	\end{equation}
By Lemma \ref{lem:210pm23jun21},
	\begin{equation}
	\begin{aligned}
	\mathbb{P}(V_{\sss T}(G_{n,p_n})\geq k_n) &\geq C\frac{n(n-1)\cdots(n-k_n +1)}{(3!)^{\frac{k_n}{3}} 
	(\tfrac13 k_n)!}\, p_n^{k_n}(1-p_n)^{{k_n\choose 2}-k_n+ (n-k_n)k_n} \\
	&=C\frac{n!}{(n-k_n)!(3!)^{\frac{k_n}{3}} 
	(\tfrac13 k_n)!}\, p_n^{k_n}(1-p_n)^{{k_n\choose 2}-k_n+ (n-k_n)k_n}
	\end{aligned}
	\end{equation}
for some constant $C>0$. Using Stirling's approximation, 
	\eqn{
	\label{Stirling}
	\sqrt{2\pi}L^{L+{\frac{1}{2}}} \mathrm{e}^{-L}\leq L! \leq \mathrm{e}L^{L+{\frac{1}{2}}}\mathrm{e}^{-L},
	}
we get
        \begin{equation}
	\begin{aligned}
	\mathbb{P}(V_{\sss T}(G_{n,p_n})\geq k_n) &\geq C\frac{n^{n+\frac{1}{2}} \mathrm{e}^{k_n}}{(n-k_n)^{n-k_n+\frac{1}{2}}(3!)^{\frac{k_n}{3}} 
	(\tfrac13 k_n)!}\, p_n^{k_n}(1-p_n)^{{k_n\choose 2}-k_n+ (n-k_n)k_n} \\
	&\geq C\frac{n^{k_n} \mathrm{e}^{k_n}}{(3!)^{\frac{k_n}{3}} 
	(\tfrac13 k_n)!}\, p_n^{k_n}(1-p_n)^{{k_n\choose 2}-k_n+ (n-k_n)k_n}.
	\end{aligned}
	\end{equation}
Using Stirling's approximation in \eqref{Stirling} again, and the fact that $\log(1-x)\geq -2x$ for $0<x<\tfrac12$, we can write
	\begin{equation}
	\begin{aligned}
	\mathbb{P}(V_{\sss T}(G_{n,p_n})\geq k_n) &\geq C \exp\left(k_n \log{n} -Ck_n -k_n \log{n}-\tfrac{2}{n}(n-k_n)k_n\right) \\
	&\qquad\times \exp\left(-\tfrac13 k_n \log(\tfrac13 k_n)\right).
	\end{aligned}
	\end{equation}
Therefore, possibly after enlarging $C$, arrive at
	\begin{equation}
	\mathbb{P}(V_{\sss T}(G_{n,p_n})\geq k_n)\geq C \exp\left(-\tfrac13 k_n \log(\tfrac13 k_n)- Ck_n\right)
	\end{equation}
as required.
\end{proof}


\subsection{Proof of the upper bound in Theorem \ref{thm-many-vertices-in-triangles}}
\label{S7.2} 

In this section, we prove the upper bound on $\mathbb{P}(V_{\sss T}(G_{n,p_n})\geq k_n)$ stated in Theorem \ref{thm-many-vertices-in-triangles}:

\begin{thm}{\bf [Upper bound on the probability of many vertices in triangles]}
\label{thm:520pm10sep21}
$\mbox{}$\\
Let $(k_n)_{n\in\N}$ be such that $\lim_{n \to \infty}\frac{k_n}{\log n} = \infty$. Then, there exists a constant $C>0$ such that, for $n$ large enough,
	\begin{equation}
	\log	\mathbb{P}(V_{\sss T}(G_{n,p_n})\geq k_n) \leq  -\tfrac13 k_n\log(\tfrac13k_n)+Ck_n.
	\end{equation}
\end{thm}

First we develop some combinatorial tools that we will need along the way. We begin with a simple lemma that provides a lower bound on the number of edges for a graph that has $q$-vertices in triangles:

\begin{lem}{\bf [Edges and triangles]}
\label{lem:720pm22may21}
Let $G$ be a graph with $V_{\sss T}(G) = q$. Then $|E(G)| \geq q$.
\end{lem}

\begin{proof}
Let $d_v$ denote the degree of vertex $v\in[n]$. If $v$ is part of a triangle, then $d_v \geq 2$. Let $v_1,v_2,\ldots,v_q$ be the vertices that are part of a triangle. Then
	\begin{equation}
	2|E(G)|=\sum_{v\in V(G)}d_v \geq \sum_{j=1}^q d_{v_j} \geq 2q.
	\end{equation} 
\end{proof}

We find it more convenient to work with a basic subgraph of a graph with at least $q$ vertices that are part of a triangle:

\begin{dfn}{\bf [$q$-basic graph]}
\label{def-q-basic-subgraph}
	{\rm A subgraph $G \subseteq K_n$ is called \emph{$q$-basic} if $V_{\sss T}(G) = q$ and $V_{\sss T}(G\setminus e) < q$ for all edges $e \in G$. Here, $G\setminus e$ denotes the graph with the edge $e$ removed.}\hfill$\spadesuit$
\end{dfn}

\begin{rmk}{\bf [Lower bound on edges in $q$-basic graphs]}
\label{rmk:311pm12oct21}
Note that every edge in a $q$-basic graph is part of some triangle, otherwise we can delete the edge without removing any triangle. \hfill$\blacksquare$
\end{rmk}

The next lemma shows that whenever there are $q$-vertices in a graph $G$ that are part of a triangle, then it contains a $q$-basic subgraph:

\begin{lem}{\bf [There is a $q$-basic graph spanning the vertices in triangles]}
\label{lem:706pm24may21}
$\mbox{}$\\
If $V_{\sss T}(G)=q$, then there is a subgraph $G'\subseteq G$ that is a $q$-basic graph.
\end{lem}

\begin{proof} 
Starting with the graph $G$, we delete an edge $e$ such that $V_{\sss T}(G\setminus e)=q$. We repeat this until there is no such edge left. Let us denote this graph by $G'$. At this point, by our construction, $V_{\sss T}(G)=q$ but $V_{\sss T}(G'\setminus e)<q$ for every $e\in G'$. Therefore, by Definition \ref{def-q-basic-subgraph}, $G'$ a $q$-basic graph.
\end{proof}

The following is our key decomposition lemma of a $q$-basic graph, which will be seen to be instrumental in counting their number. This lemma gives us the rough structure of a $q$-basic graph. For a graph $G=(V(G),E(G))$ and $u,v,w \in V(G)$, $w$ is called a {\em co-neighbor} of a pair of vertices $u$ and $v$ if $(u,w),(v,w)\in E(G)$. Also, for $U\subseteq V$, $G[U]$ is the induced subgraph of $G$ on the vertex set $U$. Finally, $U \subseteq V$ is called an {\em independent set} if there is no edge in $G[U]$.

\begin{lem}{\bf [Decomposition lemma of $q$-basic graphs]}
\label{lem:219pm23may21}
The vertices of a $q$-basic graph $G$ can be partitioned into three disjoint parts $V_1, V_2, V_3$ such that the following properties hold:
\begin{enumerate}
\item[--]  $G[V_1]$ is a union of vertex-disjoint triangles.
\item[--] $G[V(G) \setminus V_1]$ does not contain a triangle.
\item[--] $G[V_2]$ is a union of vertex-disjoint edges, and the endpoints of each edge in $G[V_2]$ have a co-neighbour in $V_1$.
\item[--] $G[V_3]$ is an independent set, and each vertex in $V_3$ is either a co-neighbour of the endpoints of an edge in $V_1$, or a co-neighbour of the endpoints of an edge between $V_1$ and $V_2$.
\end{enumerate} 
\end{lem}

\begin{proof}
First, we extract the vertex-disjoint triangles from $G$, one by one in an arbitrary order, and stop when no more triangles can be extracted. The set of vertices of the extracted triangles is denoted by $V_1$. Clearly, there is no triangle left in the graph induced by $G[V(G) \setminus V_1]$. 

Next, starting from $G[V(G) \setminus V_1]$, we extract vertex-disjoint edges, one by one in an arbitrary order, and stop when no vertex-disjoint edge is left. The set of vertices in the extracted edges is denoted by $V_2$. Note that every endpoint of an edge in $G[V_2]$ has a co-neighbour in $V_1$, because otherwise the edge is not part of any triangle, which is impossible by Remark \ref{rmk:311pm12oct21}.  

Finally, put $V_3:=V(G) \setminus(V_1\cup V_2)$. Note that $G[V_3]$ is an independent set, because otherwise there is an edge in $G[V_3]$ whose endpoints, by construction, would have been added to $V_2$. Also, each vertex is in a triangle and therefore each vertex in $G[V_3]$ is either a co-neighbour of the endpoints of an edge in $V_1$, or a co-neighbour of the endpoints of an edge between $V_1$ and $V_2$. Indeed, the other choice is impossible: if a vertex in $G[V_3]$ is a co-neighbour of the endpoints of an edge in $V_2$, then we get a triangle outside $V_1$, which violates our construction.
\end{proof}

Lemma \ref{lem:219pm23may21} gives us a natural way to decompose the $q$-basic graphs in terms of the sizes of $V_1, V_2$ and $V_3$.

\begin{dfn}{\bf [$(\ell_1,\ell_2,\ell_3)$ configuration of $q$-basic graphs]}
\label{def-ell-configuration}
{\rm A $q$-basic graph $G$ has an $(\ell_1,\ell_2,\ell_3)$ configuration if there is a decomposition as in Lemma \ref{lem:219pm23may21} with $|V_i|=\ell_i$ for $i=1,2,3$. (Note that $q=\ell_1+\ell_2+\ell_3$.)}\hfill $\spadesuit$
\end{dfn}

Definition \ref{def-ell-configuration} is useful to upper bound the number of edges in $q$-basic graphs:

\begin{lem}{\bf [Upper bound on edges in $q$-basic graphs]}
\label{lem:603pm23may21}
A $q$-basic graph has at most $3q$ edges.
\end{lem}

\begin{proof} 
Suppose that the $q$-basic graph has an $(\ell_1,\ell_2,\ell_3)$ configuration. Then $G[V_1]$ has $\ell_1/3$ disjoint triangles, which contribute $\ell_1$ edges. Note that these edges put all vertices in $V_1$ in some triangle.

Next, $\ell_2/2$ edges in $G[V_2]$ are due to its $\ell_2/2$ disjoint edges. Since, by Lemma \ref{lem:219pm23may21}, the pair of endpoints of each edge in $V_2$ has a co-neighbor in $V_1$, they contribute another $\ell_2$ edges. Therefore in total we get ${3\ell_2}/2$ edges, and all vertices in $V_2$ are now part of some triangle.

Finally, since each vertex in $V_3$ is an independent set, and all vertices are in some triangle, the maximum number of edges needed to put each vertex in a triangle is $3\ell_3$. At this point all vertices in $V_3$ are part of some triangle as well.

Since the graph is $q$-basic, there are no more edges present in the graph (recall Remark \ref{rmk:311pm12oct21}). Therefore the graph has at most $\ell_1+\tfrac32 \ell_2+3\ell_3$ edges. Since $\ell_1+\ell_2+\ell_3=q$, we arrive at the conclusion that a $q$-basic graph has at most $3q$ edges.
\end{proof}

The next lemma counts the number of $q$-basic graphs with configuration $(\ell_1,\ell_2,\ell_3)$.

\begin{lem}{\bf [The number of $q$-basic graphs in configuration]}
\label{lem:1149pm24may21}
The number of $q$-basic subgraphs of $K_n$ with configuration $(\ell_1,\ell_2,\ell_3)$ is at most 
	\begin{equation}
	\tfrac{n^{q}} {(3!)^{\tfrac13 \ell_1} \times (\tfrac13 \ell_1)!\times (2!)^{{\tfrac12 \ell_2}}\times (\tfrac12 \ell_2)! 
	\times \ell_3!} \times \ell_1^{\tfrac12 \ell_2} \times \left(3q\right)^{\ell_3}.
	\end{equation}
\end{lem}

\begin{proof}
Using Lemma \ref{lem:219pm23may21}, we get that the number of $q$-basic subgraphs of $K_n$ with configuration $(\ell_1,\ell_2,\ell_3)$ is at most
	\begin{equation}
	\tfrac{n(n-1)\cdots(n-\ell_1-\ell_2-\ell_3+1)}{(3!)^{\tfrac13 \ell_1} \times (\tfrac13 \ell_1)! 
	\times (2!)^{{\tfrac12 \ell_2}} \times (\tfrac12 \ell_2)! \times \ell_3!}  \ell_1^{\tfrac12 \ell_2}  
	\left(|E\left(V_1\right)| + |E\left(V_1,V_2\right)|\right)^{\ell_3}.
	\end{equation}
First, let us explain the combinatorial factor 
	\begin{equation}
	\tfrac{n(n-1)\cdots(n-\ell_1-\ell_2-\ell_3+1)}{(3!)^{\tfrac13 \ell_1} \times (\tfrac13 \ell_1)! 
	\times (2!)^{{\tfrac12 \ell_2}} \times (\tfrac12 \ell_2)! \times \ell_3!}.
	\end{equation}
The numerator counts the number of ways to select the $q=\ell_1+\ell_2+\ell_3$ vertices, the term $ (\tfrac13 \ell_1)!$ in the denominator counts the number of permutation of the $\ell_1/3$ triangles in $V_1$, while $(3!)^{\tfrac13 \ell_1}$ counts the permutations of the vertices of each of the individual triangles. Similarly, $(\tfrac12 \ell_2)!$ accounts for the permutation of the edges in $V_2$, and $(2!)^{{\tfrac12 \ell_2}} $ accounts for the permutations of the vertices of each of the individual edges. Finally, $\ell_3!$ counts the permutations of the vertices in $V_3$.

Second, $\ell_1^{\tfrac12 \ell_2}$ is the number of ways the endpoints of the $\ell_2/2$ vertex-disjoint edges can select a co-neighbor from $V_1$, and $\left(|E\left(V_1\right)| + |E\left(V_1,V_2\right)|\right)^{\ell_3}$ is the number of ways a vertex in $V_3$ can be part of some triangle, i.e., each vertex in $V_3$ is either a co-neighbor of an edge in $G[V_1]$ or a co-neighbor of an edge between $V_1$ and $V_2$.

Using Lemma \ref{lem:603pm23may21}, this can be bounded from above by
	\begin{equation}
	\tfrac{n^{\ell_1+\ell_2+\ell_3}} {(3!)^{\tfrac13 \ell_1} \times (\tfrac13 \ell_1)!\times (2!)^{{\tfrac12 \ell_2}}
	\times (\tfrac12 \ell_2)! \times \ell_3!} \times \ell_1^{\tfrac12 \ell_2} \times \left(3q\right)^{\ell_3}.
	\end{equation}
\end{proof}

Before investigating the number of $q$-basic graphs with $m$ edges, we study a minimization problem that appears in it:

\begin{lem}{\bf [Optimization problem]}
\label{lem:311pm23jun21}
For $q$ sufficiently large, the minimum value of the function 
	\begin{equation}
	f(x_1,x_2,x_3) = \tfrac13 x_1 \log(\tfrac13 x_1) + \tfrac12 x_2 \log(\tfrac12 x_2)+ x_3\log x_3
	\end{equation}
subject to the constraint $x_1+x_2+x_3=q$ is at least $\tfrac13 (q-q^{2/3}\log{q})\log(\tfrac13 (q-q^{2/3})\log q)$.
Further, the maximizers satisfy $q-q^{2/3}\log q \leq x_1 \leq q$ and $0 \leq x_2+x_3 \leq q^{2/3} \log q$.
\end{lem}

\begin{proof}
We use the Lagrange multiplier method. Define
	\begin{equation}
	L := \tfrac13 x_1 \log(\tfrac13 x_1) + \tfrac12 x_2 \log(\tfrac12 x_2)+ x_3\log x_3 + \mu(q-x_1-x_2-x_3).
	\end{equation}
Differentiating $L$ with respect to $x_1,x_2,x_3$, we get
	\begin{equation}
	\frac{\partial L}{\partial x_1} = \tfrac13 \log(\tfrac13 x_1) + \tfrac13 - \mu,\,\, 
	\frac{\partial L}{\partial x_2} = \tfrac12 \log(\tfrac12 x_2) + \tfrac12 - \mu,\,\, 
	\frac{\partial L}{\partial x_3} = \log x_3 + 1 -\mu,
	\end{equation}
with $\mu$ the Lagrange multiplier. Setting these derivatives equal to zero, and substituting them into $x_1+x_2+x_3=q$, we get
	\begin{equation}
	\label{eqn:627pm22jun21}
	3\mathrm{e}^{3\mu-1}\left(1+\tfrac23 \mathrm{e}^{-\mu}+\tfrac13 \mathrm{e}^{-2\mu}\right)=q.
	\end{equation}
Note that the left-hand side of \eqref{eqn:627pm22jun21} is strictly increasing in $\mu$, and therefore there is a unique solution to \eqref{eqn:627pm22jun21}. Moreover, the solution must be less than $\mu_{U}$, where $3\exp{\left(3\mu_U-1\right)} = q$. 

Next, we show that when $q$ is large enough, the solution is larger than $\mu_{L}$, where $3\exp{\left(3\mu_L-1\right)} = q-q^{2/3}\log{q}$. To see why, substitute $\mu_L$ into \eqref{eqn:627pm22jun21}, to get that the left-hand side is equal to
	\begin{equation}
	\left(q-q^{2/3}\log{q}\right)\left(1+\tfrac{2}{3}\left(\tfrac{3}{(q-q^{2/3}\log{q})\mathrm{e}}\right)^{1/3}
	+\tfrac{1}{3}\left(\tfrac{3}{(q-q^{2/3}\log{q})\mathrm{e}}\right)^{2/3}\right).
	\end{equation}
Now using the fact that $q-q^{2/3}\log{q}\geq 0.9q$ for $q$ sufficiently large, for some constant $C$, we see that the last display is bounded above by
	\begin{equation}
	\left(q-q^{2/3}\log{q}\right)\left(1+\frac{C}{q^{1/3}}+\frac{C}{q^{2/3}}\right)<q. 
	\end{equation} 
Hence, for $q$ sufficiently large, the solution of \eqref{eqn:627pm22jun21} lies in the interval 
	\begin{equation}
	(\mu_L,\mu_U) = \left(\tfrac13\log(\tfrac13(q-q^{2/3}\log{q}))+\tfrac13, 
	\tfrac13 \log\tfrac12 q + \tfrac13\right).
	\end{equation}
Therefore $q-q^{2/3}\log q \leq x_1 \leq q$, and consequently $0 \leq x_2+x_3 \leq q^{2/3} \log q$. Hence the minimum value of $f$ is at least $\tfrac13 (q-q^{2/3}\log q)\log(\tfrac13 (q-q^{2/3}\log q))$.
\end{proof}

The next lemma bounds the number of $q$-basic graphs in terms of the number of its edges.

\begin{lem}{\bf [The number of $q$-basic graphs in terms of edges]}
\label{lem:459pm24may21}
For $q$ sufficiently large, the number of $q$-basic subgraphs of $K_n$ with $m$ edges is at most 
	\begin{equation}
	C(3q)^3 \exp\left(m\log n -\tfrac13 q \log(\tfrac13 q)+16q\right)
	\end{equation}
for some constant $C>0$.
\end{lem}

\begin{proof}
First note that the number of edges in a $q$-basic graph with configuration $(\ell_1,\ell_2,\ell_3)$ is at least 
	\begin{equation}
	\ell_1 +\tfrac32 \ell_2 + 2\ell_3.
	\end{equation}
To see why, observe that $\ell_1$ edges in $G[V_1]$ are contributed by $\ell_1/3$ disjoint triangles, and $\ell_2/2$ edges in $G[V_2]$. Since the endpoints of each edge in $V_2$ have a co-neighbour in $V_1$, they contribute another $\ell_2$ edges. Moreover, since each vertex in $V_3$ is either a co-neighbour of the endpoints of an edge in $V_1$ or a co-neighbour of the endpoints of an edge between $V_1$ and $V_2$, they contribute at least $2\ell_3$ edges. Therefore, via Lemma \ref{lem:1149pm24may21}, the number of $q$-basic graphs with $m$ edges is at most
	\begin{equation}
	\sum_{\ell_1+\frac{3}{2}\ell_2+2\ell_3\leq m}\tfrac{n^{q}} { (3!)^{\tfrac13 \ell_1} \times
	(\tfrac13 \ell_1)!\times (2!)^{{\tfrac12 \ell_2}}\times (\tfrac12 \ell_2)! \times \ell_1!} \times \ell_1^{\tfrac12 \ell_2} 
	\times \left(3q\right)^{\ell_3}.
	\end{equation}
Since $\ell_1+\frac{3}{2}\ell_2+2\ell_3 \leq m$ and $\ell_1+\ell_2+\ell_3=q$, we get that $q\leq m-\tfrac12 \ell_2-\ell_3$. Further, using $m\leq3q$, we get that there are at most $(3q)^3$ terms in the sum. Therefore we can bound the expression in the last display crudely by
	\begin{equation}
	(3q)^3\, \tfrac{\exp\left((m-\tfrac12 \ell_2-\ell_3)\log{n}+\tfrac12 \ell_2\log{\ell_1}+\ell_3\log{3q}\right)} 
	{(3!)^{\tfrac13 \ell_1} \times (\tfrac13 \ell_1)!\times (2!)^{{\tfrac12 \ell_2}}\times (\tfrac12 \ell_2)! \times \ell_3!}.
	\end{equation}
Using that $\ell_1, q\leq n$, we see that this can by bounded by
	\eqan{
	\label{eqn:419pm24may21}
	&(3q)^3 \tfrac{\exp\left(m\log{n}+\ell_3\log 3\right)} {(3!)^{\tfrac13 \ell_1} \times (\tfrac13 \ell_1)!
	\times (2!)^{{\tfrac12 \ell_2}}\times (\tfrac12 \ell_2)! \times \ell_3!} \\
	&\leq C(3q)^3\exp\left(m\log{n}-\left(\tfrac13 \ell_1+\tfrac12\right)\log{\tfrac13 \ell_1}
	+4\ell_1-\left(\tfrac12 \ell_2+\tfrac12\right)\log(\tfrac12 \ell_2)+{\ell_2}-\left(\ell_3+\tfrac12\right)\log{\ell_3}+3\ell_3\right)\nn
	}
for some constant $C$. We can simplify \eqref{eqn:419pm24may21} further by enlarging $C$, and using the bound $\ell_i\leq q$ for the smaller order terms. This gives the bound
	\begin{equation}
	\exp\left(m\log n-\tfrac13 \ell_1\log(\tfrac13 \ell_1)-\tfrac12 \ell_2\log(\tfrac12 \ell_2)-\ell_3\log{\ell_3}+15q\right),
	\end{equation}
where we use Stirling's approximation in \eqref{Stirling}. Now note that in \eqref{eqn:419pm24may21} the term $\frac13 \ell_1 \log(\tfrac13\ell_1)$ is asymptotically the smallest (see Lemma \ref{lem:311pm23jun21} for the formal proof), i.e., subject to  $\ell_1+\ell_2+\ell_3={q}$ the minimum value of $\tfrac13 \ell_1\log(\tfrac13 \ell_1)+\tfrac12 \ell_2\log(\tfrac12 \ell_2)+\ell_3\log{\ell_3}$ is bounded below by $\tfrac13 (q-q^{2/3}\log{q})\log(\tfrac13 (q-q^{2/3})\log q)$ for $q$ sufficiently large $q$. Therefore \eqref{eqn:419pm24may21} can be bounded from above by
	\begin{equation}
	C(3q)^3 \exp\left(m\log n - \tfrac13 q \log(\tfrac13 q)+16q\right),
	\end{equation}
as required.
\end{proof}

We next estimate the probability that $G_{n,p_n}$ contains a $q$-basic subgraph of $K_n$:

\begin{lem}{\bf [Upper bound on the probability to contain a $q$-basic subgraph]}
\label{lem:540pm24may21}
There exists a constant $C>0$ such that, for all sufficiently large integers $q$,
	\begin{equation}
	\mathbb{P}(G_{n,p_n} \text{\rm ~contains a }q\text{-basic graph of } K_n) 
	\leq \exp\left(-\tfrac13 q \log(\tfrac13 q)+Cq\right).
	\end{equation} 
\end{lem}

\begin{proof}
Using Lemmas \ref{lem:720pm22may21}, \ref{lem:603pm23may21} and \ref{lem:459pm24may21}, we get
	\begin{equation}
	\begin{aligned}
	&\mathbb{P}(G_{n,p_n} \text{ contains a } q\text{-basic subgraph of } K_n)\\ 
	&= \sum_{m=q}^{3q}\mathbb{P}(G_{n,p_n} \text{ contains a }q\text{-basic subgraph of } K_n \text{ with } m\text{ edges}) \\
	&= \sum_{m= q}^{3q} p_n^{m}\, \#\{q\text{-basic subgraphs of } K_n \text{ with } m \text{ edges}\} \\
	&= C(3q)^2 \sum_{m= q}^{3q} \exp\left(m\log\lambda -m\log n +m\log n-\tfrac13 q\log(\tfrac13 q)+16q\right) \\
	&\leq \exp\left(-\tfrac13q \log(\tfrac13q)+Cq\right)
	\end{aligned}
	\end{equation}
for some constant $C>0$.
\end{proof}

We are now ready to estimate $\mathbb{P}(V_{\sss T}(G_{n,p_n})\geq k_n)$ as in Theorem \ref{thm:520pm10sep21}:

\begin{proof}[Proof of Theorem \ref{thm:520pm10sep21}] Without loss of generality, we may take $a=1$ in the proof below. Using Lemma \ref{lem:706pm24may21} and the union bound, we get
	\begin{equation}
	\label{eqn:539pm24may21}
	\begin{aligned}
	\mathbb{P}(V_{\sss T}(G_{n,p_n})\geq k_n) 
	&= \sum_{q= k_n}^{n} \mathbb{P}(V_{\sss T}(G_{n,p_n})= q) \\
	&\leq  \sum_{q= k_n}^{n}\mathbb{P}(G_{n,p_n} \text{ contains a } q\text{-basic subgraph of } K_n).
	\end{aligned}
	\end{equation}
Using Lemma \ref{lem:540pm24may21}, we bound \eqref{eqn:539pm24may21} from above by
	\begin{equation}
	\sum_{q= k_n}^{n} \exp\left(-\tfrac13 q\log(\tfrac13 q)+Cq\right)
	\leq n\exp\left(-\tfrac13k_n\log(\tfrac13k_n)+Ck_n\right),
	\end{equation}
as required.
\end{proof}


\section{Phase transition in exponential random graphs: Proof of Corollary \ref{cor-ERRG-vertices-in-triangles}}
\label{S8}

In this section, we apply Theorem \ref{thm-many-vertices-in-triangles} to analyze the exponential random graph model $\mathbb{P}_{\beta}$ defined in \eqref{ERRG-vertices} by proving Corollary \ref{cor-ERRG-vertices-in-triangles}.

\proof[Proof of Corollary \ref{cor-ERRG-vertices-in-triangles}] We rely on Theorem \ref{thm-many-vertices-in-triangles} and simple bounds. We first note that, by Jensen,
        \eqn{\label{eqn:127pm05dec21}
	\log Z_n(\beta\log n)=\log\mathbb{E}[\mathrm{e}^{\betanV V_{\sss T}}]\geq \mathbb{E}[\betanV V_{\sss T}]=o(n\log{n}),
	}
since $\mathbb{P}(1 \text{ is in a triangle})=o(1)$. This proves the lower bound for $\beta\leq \tfrac{1}{3}$. Further, by Theorem \ref{thm-many-vertices-in-triangles} we can lower bound
\eqn{
	\label{eqn:511pm03nov21}
	\log Z_n(\beta\log n)\geq \log\mathbb{E}[\mathrm{e}^{\betanV V_{\sss T}}\mathbbm{1}_{\{V_{\sss T}=n\}}]
	=\beta n\log{n} +\log{\mathbb{P}(V_{\sss T}\geq n)}\geq (1+o(1)) (\beta-\tfrac{1}{3})n\log{n}.
}
 This proves the lower bound $\beta>\tfrac{1}{3},$ and completes the proof of the lower bound.

For the upper bound, we estimate
\eqn{
	\label{eqn:1209pm01nov21}
	Z_n(\beta\log n)=\sum_{k=0}^n \mathbb{P}(V_{\sss T}=k)\mathrm{e}^{\betanV k}
	\leq n\max_{k\in [n]} \mathbb{P}(V_{\sss T}\geq k)\mathrm{e}^{\betanV k}.
}
Using Theorem \ref{thm-many-vertices-in-triangles}, we get that for $\beta\leq \tfrac{1}{3}$ the maximum is achieved for $k=k_n=o(n)$, while for $\beta>\tfrac{1}{3}$ the maximum is achieved for $k=n$. To see this in more detail, note that when $k=k_n=o(n)$, we can simply bound $\mathbb{P}(V_{\sss T}\geq k_n)\mathrm{e}^{\betanV k_n} \leq \mathrm{e}^{\beta k_n\log{n} }$, while when $k_n=cn$ for some constant $c\in(0,1]$ we can use Theorem \ref{thm-many-vertices-in-triangles}.

We next proceed to prove the second part of the corollary in two cases:

\medskip\noindent
\textbf{Case I:} $\beta<\tfrac13$. By \eqref{eqn:127pm05dec21} and using the definition of $\mathbb{P}_{\beta \log{n}}$, we get
\begin{equation}
	\label{eqn:354pm03nov21}
	\mathbb{P}_{\beta \log{n}}\left(\frac{V_{\sss T}}{n}\geq \varepsilon\right)
	=\frac{1}{Z_n(\beta \log n)} \mathbb{E}\left(\mathbbm{1}_{\{{V_{\sss T}}\geq \varepsilon n\}} 
	\mathrm{e}^{\beta \log{n}V_{\sss T}}\right)\leq  \frac{1}{Z_n(\beta \log n)} 
	\sum_{k\geq \varepsilon n} 
	\mathrm{e}^{\beta k\log{n}}\mathbb{P}(V_{\sss T}=k).
\end{equation}
Using Theorem \ref{thm-many-vertices-in-triangles}, we can bound
\begin{equation*}
\sum_{k\geq \varepsilon n} \mathrm{e}^{\beta k\log{n}}\mathbb{P}(V_{\sss T}=k) 
\leq n \max_{k\geq \varepsilon n}\exp{\left(\beta k\log{n}-\tfrac13 k\log(\tfrac13k)+Ck\right)}.
\end{equation*}
Since $\beta<\tfrac13$, the maximum is attained at $\varepsilon n$ and therefore, using  \eqref{eqn:127pm05dec21}, we get that \eqref{eqn:354pm03nov21} is bounded above by
\begin{equation}
	\mathbb{P}_{\beta \log{n}}\left(\frac{V_{\sss T}}{n}\geq \varepsilon\right) \leq  
	\mathrm{e}^{o(n\log{n})} 
	\mathrm{e}^{\beta \varepsilon n\log{n} - \frac{ \varepsilon n\log{n}}{3}+Cn},
\end{equation}
for some constant $C>0$. Since $\beta<\tfrac13$, the right-hand side tends to zero as $n \rightarrow \infty$.

\medskip\noindent
\textbf{Case II:} $\beta>\tfrac13$. We can write
\begin{equation}
	\label{eqn:523pm03nov21}
	\begin{aligned}
	\mathbb{P}_{\beta \log{n}}\left(\frac{V_{\sss T}}{n}<1- \varepsilon\right) &= \frac{1}{Z_n(\beta \log n)} 
	\mathbb{E}\left(\mathbbm{1}_{\{{V_{\sss T}}<(1- \varepsilon) n\}} \mathrm{e}^{\beta \log{n}V_{\sss T}}\right)\\ 
	&= \frac{1}{Z_n(\beta \log n)} \sum_{q< (1-\varepsilon) n} \mathrm{e}^{\beta k\log{n}}\mathbb{P}(V_{\sss T}=k).
	\end{aligned}
\end{equation}
Observe that
\begin{equation}
	  \sum_{q< (1-\varepsilon) n} \mathrm{e}^{\beta k\log{n}}\mathbb{P}(V_{\sss T}=k) 
	  \leq n \max_{k<(1-\varepsilon) n}\exp{\left(\beta k\log{n}-\tfrac13 k\log(\tfrac13k)+Ck\right)}. 
\end{equation}
Since $\beta>\tfrac13$, the maximum is attained at $k=(1-\varepsilon) n$. Using \eqref{eqn:511pm03nov21} we get 
	\eqn{
	Z_n(\beta \log n) \geq \exp((\beta-\tfrac13) n\log{n}(1+o(1))),
	}
and therefore we can upper bound \eqref{eqn:523pm03nov21} by
\begin{equation}
	\begin{aligned}
		\mathbb{P}_{\beta \log{n}}\left(\frac{V_{\sss T}}{n}<1- \varepsilon\right)
		& \leq \exp{\left(-(\beta-\tfrac13)n\log{n}(1+o(1))+\beta(1-\varepsilon)n\log{n}-(1-\varepsilon)\frac{n}{3}\log{n} +Cn\right)} \\
		&=\exp{\left(-(\beta-\tfrac13)o(1)n\log{n}-\beta\varepsilon n\log{n}+\varepsilon \frac{n}{3}\log{n}+Cn\right)}
	\end{aligned}
\end{equation}
for some constant $C>0$. Since $\beta>\tfrac13$, the right-hand side tends to zero as $n\rightarrow \infty$.
\qed

\medskip
\paragraph{\bf Acknowledgments}
The work in this paper was supported by the Netherlands Organisation for Scientific Research (NWO) through Gravitation-grant NETWORKS-024.002.003. The authors thank Twan Koperberg for comments on the polynomial representation of $V_{\sss T}$ in terms of the adjacency matrix, and Charles Bordenave for pointing us to reference \cite{backhausz2021typicality} and including Remark 4.12 in that reference. We would like to thank Jean-Pierre Eckmann, Svante Janson, and Yufei Zhao for their comments and pointing out several relevant references. We would also like to thank I-Hsun Chen for pointing out a minor oversight in the proof of Theorem 2.1 in the previous version of this article.


\bibliographystyle{plain}
\bibliography{references.bib}


\end{document}